\documentclass[reqno]{amsart}
\usepackage{hyperref} 
\UseRawInputEncoding
\usepackage{color}
\usepackage{latexsym}
\usepackage{graphicx}
\usepackage{mathrsfs}
\usepackage{amssymb}
\usepackage{amsfonts}
\usepackage{amssymb,amsmath,amsthm,tikz}

\newcommand{\phip}{\phi_{P^o} }

\hypersetup{colorlinks = true, urlcolor = black}
\newcommand{\red}{\color{red}}

\newcommand{\edit}[1]{ {\red \cs #1 \cs}}
\newcommand{\cs}{$\clubsuit$}

\DeclareMathOperator{\Vol}{Vol}
\DeclareMathOperator{\Ric}{Ric}
\DeclareMathOperator{\Hilb}{Hilb}

\newcommand{\Szego}{Szeg\"o }

\newcommand{\C}{\mathbb{C}}

\newcommand{\CP}{\mathbb{CP}}

\newcommand{\R}{\mathbb{R}}

\newcommand{\Z}{\mathbb{Z}}

\newcommand{\N}{\mathbb{N}}

\newcommand{\dcal}{\mathcal{D}}

\newcommand{\fcal}{\mathcal{F}}

\newcommand{\kcal}{\mathcal{K}}
\newcommand{\lcal}{\mathcal{L}}
\newcommand{\mcal}{\mathcal{M}}
\newcommand{\ncal}{\mathcal{N}}
\newcommand{\ocal}{\mathcal{O}}
\newcommand{\pcal}{\mathcal{P}}
\newcommand{\qcal}{\mathcal{Q}}

\newcommand{\wt}{\widetilde}

\newcommand{\bma}{\begin{bmatrix}}
\newcommand{\ema}{\end{bmatrix}}
\newcommand{\baa}{\begin{align*}}
\newcommand{\eaa}{\end{align*}}
\newcommand{\bea}{\begin{eqnarray*} }
\newcommand{\eea}{\end{eqnarray*} }
\newcommand{\bee}{\begin{eqnarray} }
\newcommand{\eee}{\end{eqnarray} }
\newcommand{\be}{\begin{equation} }
\newcommand{\ee}{\end{equation} }
\newcommand{\bt}{\begin{theorem}}
\newcommand{\et}{\end{theorem}}
\newcommand{\bpf}{\begin{proof}}
\newcommand{\epf}{\end{proof}}
\newcommand{\bl}{\begin{lem}}
\newcommand{\el}{\end{lem}}
\newcommand{\bc}{\begin{cor}}
\newcommand{\ec}{\end{cor}}
\newcommand{\bd}{\begin{defn}}
\newcommand{\ed}{\end{defn}}
\newcommand{\bcs}{\begin{cases}}
\newcommand{\ecs}{\end{cases}}
\newcommand{\bex}{\begin{example}}
\newcommand{\eex}{\end{example}}
\newcommand{\brem}{\begin{rem}}
\newcommand{\erem}{\end{rem}}

\newcommand{\ot}{\otimes}
\newcommand{\half}{\frac{1}{2}}
\renewcommand{\d}{\partial}
\newcommand{\dbar}{\bar\partial}
\newcommand{\ddbar}{\partial\dbar}

\def\XXint#1#2#3{{\setbox0=\hbox{$#1{#2#3}{\int}$ }
\vcenter{\hbox{$#2#3$ }}\kern-.6\wd0}}

\newtheorem{theo}{{\sc Theorem}}[section]
\newtheorem{cor}[theo]{{\sc Corollary}}

\newtheorem{lem}[theo]{{\sc Lemma}}
\newtheorem{prop}[theo]{{\sc Proposition}}
\newtheorem{defn}[theo]{{\sc Definition}}
\newtheorem{rem}{{\sc Remark}}
\newenvironment{example}{\medskip\noindent{\it Example:\/} }{\medskip}

\newcommand{\QQ}{\mathcal{Q}}

\newcommand{\T}{{\mathbf T}^m}

\newcommand{\Hess}{{\operatorname {Hess}}}

\newcommand{\szego}{Szeg\"o }

\newcommand{\kahler}{K\"ahler }

\title{Entropy of Bergman measures of a toric Kaehler manifold}

\author{Pierre Flurin and Steve Zelditch }
\address{Department of Mathematics, Northwestern  University, Evanston, IL 60208, USA}

\email{zelditch@math.northwestern.edu}

\thanks{Research partially supported by NSF grant  DMS-1810747.
and by the Stefan Bergman trust  }

\date{\today}

\begin{document}

\begin{abstract} Associated to the Bergman kernels of a polarized toric
\kahler manifold $(M, \omega, L, h)$  are  sequences of measures $\{\mu_k^z\}_{k=1}^{\infty}$
parametrized by the points $z \in M$. We determine the asymptotics of the entropies $H(\mu_k^z)$
of these measures. 
The sequence $\mu_k^z$ in some ways resembles a sequence of
convolution powers; we determine precisely when it actually is such a sequence. When $(M, \omega)$ is a Fano toric manifold with positive
Ricci curvature, we show that there  exists a unique point $z_0$ (up to the real torus action) for which $\mu_k^z$ has asymptotically
maximal entropy. If the \kahler metric is \kahler-Einstein, we show that the image of $z_0$ under the moment map is the center of mass
of the polytope. 
We also show that the Gaussian measure on the space $H^0(M, L^k)$ induced by the \kahler metric has
maximal entropy at the balanced metric. 

\end{abstract}

\maketitle

In \cite{Z09}, the second author introduced a sequence $\{\mu_k^z\}_{k=1}^{\infty} $ of  probability 
measures on the convex lattice polytope $P \subset \R^m$ associated
to a toric \kahler manifold $(M, \omega)$. The measures $\mu_k^z$ are supported on the dilated lattice points $P \cap \frac{1}{k}  \Z^m$, and depend on a choice
of Hermitian metric  $h = e^{-\varphi} $ on the toric line bundle  $L \to M$ with $\omega_{\varphi}: = i \ddbar \varphi = \omega$. They  also depend on 
a point $z \in M$, or more precisely on its image $x = \mu(z)$ under the
moment map \begin{equation} \label{MM} \mu_h:= \mu:  M \to P \subset \R^m, \end{equation}  associated to $h$. In the special
case where $M = \CP^1$ and $\omega = \omega_{FS}$ is the Fubini-Study metric, the measures $\mu_k^z$  are the standard binomial measures indexed by $x \in [0,1]$ and coincide with the $k$th convolution power $\mu_x^{* k}$  of
  the Bernoulli measure $\mu_x = x \delta_1 + (1-x) \delta_0$ on $[0,1]$.
 More generally, for the Fubini-Study metric $h_{FS}$ on the kth power of the standard line bundle  $\ocal(k) \to \CP^m$ in any dimension, the measures $\mu_k^z$ are the standard multi-nomial distributions, which are also a sequence of convolution powers. 
    For general toric \kahler manifolds, 
  the  sequences $\{\mu_k^z\}_{k=1}^{\infty}$ is certainly {\it not} a sequence
  of convolution powers.     Yet, many of the classical results on convolution powers are also valid for the sequence $\{\mu_k^z\}_{k=1}^{\infty} $: In \cite{SoZ12} they
  are shown to  satisfy  a law of large numbers and a
large deviations principle; more recently, they were proved to satisfy  a central limit theorem \cite{ZZ18}. The purpose of this
note is to given an asymptotic formula for the entropies of $\{d\mu_k^z\}_{k=1}^{\infty} $, extending the family of probabilistic results one step further. 
We further investigate the points $z$ and metrics $h$ for
which the sequences have asymptotically maximal entropy.  The proofs are non-probabilistic and   are based on Bergman kernel asymptotics, and especially on the local CLT results in \cite{ZZ18} and on the LDP in \cite{SoZ12}.

To state the result, we introduce some notation, referring to Section \ref{BACKGROUND} and to  \cite{Z09, SoZ12,ZZ18} for much of the background.
The moment map \eqref{MM}
   associated to this data defines a torus bundle of the open orbit of $(\C^*)^m$   over the interior of the  convex  lattice polytope  $P$. As reviewed in Section \ref{MONSECT}, there is a natural basis $\{s_{\alpha}\}_{\alpha \in k P}$  of the space $H^0(M, L^k)$ of holomorphic sections of the $k$-th power of $L$ by 
eigensections $s_{\alpha}$ of the $\T$ action. In a standard frame $e_L$ of $L$ over $M^o$, they correspond to monomials $z^{\alpha}$ on $(\C^*)^m$. The pointwise
norms of $z^{\alpha}$ in the open orbit are given by $|z^{\alpha}|^2 e^{- k \varphi(z)}$ where $h = e^{-\phi}$ in a standard frame. 
The toric \kahler potential  $\phi$ on the open orbit is $\T$-invariant
  and may be viewed as a convex function on $\R^m$. Its Legendre transform $u$
  is a convex function on $P$ known as the symplectic potential. For instance, the symplectic potential of the Fubini-Study 
metric is $u_{FS}(x) = x \log x + (1 -x) \log (1- x)$ (see Section \ref{BACKGROUNDOO}).

For $\alpha \in k P \cap \Z^m$, we define \begin{equation} \label{PHK} \pcal_{h^k}(\alpha, z): =
\frac{|z^{\alpha}|^2 e^{- k \varphi(z)}}{Q_{h^k}(\alpha)},
\end{equation}
where $Q_{h^k}(\alpha)$ is defined in \eqref{QFORM}.
 Further, we denote by $\Pi_{h^k}: L^2(M, L^k) \to H^0(M, L^k)$ the \szego projector and by $\Pi_{h^k}(z)$ the associated density
of states, i.e. the metric contraction of the diagonal of the kernel of $\Pi_{h^k}$; see Section \ref{MONSECT}.  We now come to the main definition:
\begin{defn}
For any $z \in M^o$ and $k \in \N$, we define the probability measure on  $P \subset \R^m$ by,
\begin{equation} \label{MUKZDEF} \mu_k^z := \frac{1}{\Pi_{h^{k }}(z,z)}\;\; \sum_{\alpha \in kP \cap \Z^m}
\frac{|s_{\alpha}(z)|_{h^{k }}^2}{\|s_{\alpha}\|_{h^{k }}^2}   \;
\delta_{\frac{\alpha}{k }} =   \frac{1}{\Pi_{h^{k }}(z,z)}\;\; \sum_{\alpha \in kP \cap \Z^m}  \pcal_{h^k}(\alpha, z) \delta_{\frac{\alpha}{k }}
\end{equation} 
 \end{defn}  Note that 
$\frac{1}{\Pi_{h^{k }}(z,z)} \sum_{\alpha \in k P \cap \Z^m} 
 \pcal_{h^k}(\alpha, z)  = 1. $
The measures are $\T$-invariant in $z$, and therefore define a family   discrete  measures  on
  $P \cap \frac{1}{k} \Z^m$ parametrized by points $\mu_h(z) \in P$. Although it is not explicit in the notation, $\mu_k^z$ depends on 
  the choice of Hermitian metric $h$ on $L$. For background on `lattice probability measures' we refer to \cite{GK}.
  
  \begin{rem} All of the techniques and result of this article can be extended to the case where $\mu(z) \in \partial P$, i.e. $z$ lies on the divisor
  at infinity of $M$. Indeed, the formulae derive from the large deviations principle of \cite{SoZ12} and the convergence theorem for
  geodesics \cite{SoZ07, SoZ10} and these results were proved for all $z$, including $z$ on the divisor at infinity. But since it is lengthier and more technical to work at the boundary, for the sake of brevity we assume $z \in M^o$ in this article. \end{rem}

  \subsection{Asymptotics of entropy  of $\mu_k^z$}
  
    The (Shannon) entropy of a discrete probability measure with masses   $\{p_{\alpha}\}$ is defined by (cf. \cite{KS})  $$H  = - \sum_{\alpha} p_{\alpha} \ln p_{\alpha}. $$
  Thus, the entropy of $\mu_k^z$ is
  \begin{equation} \label{ENTFORM} H(\mu_k^z) =- \sum_{\alpha \in k P 
  \cap \Z^m} \frac{\pcal_{h^k}(\alpha, z)}{\Pi_{h^k}(z)}  \ln \frac{\pcal_{h^k}(\alpha, z)}{\Pi_{h^k}(z)} . \end{equation}
  The asymptotic entropy  result is:

  \begin{theo} \label{H}  Let $h = e^{-\phi}$ be a toric Hermitian metric on $L \to M$ and let $\omega_{\varphi} = i \ddbar \phi$ be the corresponding \kahler metric.
   Then, as $k \to \infty$, 
  \[ H(\mu_k^z) = \frac{1}{2}\log(\det \left((2 \pi e k)  (i \ddbar \varphi \rvert_z) \right) + o(1) \]  \end{theo}

Note that  the entropy depends only on the image $\mu_h(z) = x_0$ of $z$ under the moment map \eqref{MM}. Also, $\det (i \ddbar \varphi)$ is the density of
the volume form $\omega_{\varphi}^m$ relative to Lebesgue measure on the open orbit.    
 As in \cite{Ab98}
it is convenient to rewrite $\log \det i \ddbar \varphi$  in terms of the symplectic potential and its Hessian in  action-angle variables,  with action variables $x \in P$ and angle variables $\theta$ on $\mu_h^{-1}(x)$. We recall that  the  symplectic potential $u$ is the Legendre transform of the open orbit \kahler potential; we refer Section \ref{BACKGROUNDOO}  and to \cite{Ab98, Ab03} for background. 
Then set, 
 \begin{equation} \label{LDEF} L(x) = \half \log \det \nabla^2 u(x)  = - \half \log \det i \ddbar \varphi, \end{equation}
and Theorem \ref{H} may be reformulated as follows. 
  \begin{theo} \label{H2TH}  Let $h = e^{-\phi}$ be a toric Hermitian metric on $L \to M$ and let $u$ be the open orbit symplectic potential. 
   Then, as $k \to \infty$, 
  \[ H(\mu_k^z) = \frac{1}{2}\log(\det \frac{(2 \pi e k)}{ \nabla^2 u |_{\mu_h(z)} })+ o(1) = \frac{m}{2} \log (2 \pi e k)   - L(x) + o(1). \]  \end{theo}
  
  Note that 
the entropy of uniform measure $\mu_{k P \cap \Z^m} $  on a set of $r$ element is $\log r$. The number $  \# (kP \cap \Z^m)$ of such lattice points is $\simeq k^m \#(P \cap \Z^m)$, so that uniform measure on these lattice points has entropy
   $m \log k + \log \#(P \cap \Z^m) $.  $\mu_k^z$ is not uniform, but rather    is approximately a  discretized Gaussian distribution centered at $\mu(z)$ and of width $k^{-\half}$  (see Lemma \ref{pcalLem} and  Lemma \ref{COMPAREPIT2}  for the precise statements).  A discretized 
   Gaussian of width $k^{-\frac{1}{2}}$ and of height $k^m$  is concentrated in the ball $B(z, k^{-\half})$ and is similar to uniform measure on that ball of the same
   height. This approximation accurately predicts  the leading order term  $\log k^{m/2}$.

\begin{rem} One may expect analogous results for non-compact infinite volume toric \kahler manifolds, such as $\C^m$ with the Bargmann-Fock space
of analytic functions. The techniques of \cite{F12} apply in that setting. However, the large deviations results have not been established in such cases,
and we confine the article to compact \kahler manifolds. \end{rem}

  Theorem \ref{H2TH}  specializes to  known asymptotics of entropies  of multinomial distributions when $(M, \omega)$ is complex projective space with Fubini-Study metric.
    In dimension $m =1$, the binomial distributions are convolution powers $\mu_k^p =  (\mu_p)^{*k}$ of the Bernoulli measure $\mu_p$ defined by $\mu_p(\{1\}) = p, 
  \mu_p(\{0\}) = 1-p$.   In this case, the entropy asymptotics can be obtained from local central limit theorems
  and Stirling's formula, and according to  \cite[Theorem 2]{JSz99} and to \cite{K98}), $H(\mu_k^z)$ has a complete asymptotic expansion in powers of $k^{-1}$ whose coefficients
involve the Bernoulli numbers.   The entropy of $\mu_p$  is $p \log p + (1 - p) \log (1 - p) = u_{FS}(p)$, the Fubini-Study symplectic potential (see Section 
\ref{BACKGROUNDOO} and \cite{Ab98} for background). Thus, $p (1-p) = (u_F''(p))^{-1}$. The parameter
  $p \in [0,1]$ 
  is the image of the parameter $z \in \CP^1$ under the Fubini-Study moment map. 
The $k$th convolution power $\mu_k^{p }$ is the binomial measure, for which
$p_{k, \ell} = {k\choose  \ell} p^\ell (1 -p)^{k - \ell}$. Its Shannon entropy
has the asymptotics (see \cite[Corollary 1]{JSz99}),
$$H(\mu_k^p)  = \half \log k + \half(1 + \log(2 \pi p(1 - p)) + O(k^{-\half} + \epsilon). $$  To compare with Theorem \ref{H2TH}, we note that in the Fubini-Study case, 
$u_{FS}'(x) = \log \frac{x}{1-x}$,
$u_{FS}''(x) = \frac{1 }{x (1 - x)}$,  $\log (u_{FS}''(x))^{-1} =  \log x (1-x)$.

Now consider multinomial distributions, which correspond to the toric \kahler manifold $M = \CP^{m}$
with the Fubini-Study metric $h_{FS}$ on $L = \ocal(1)$. The parameters
  $\vec p$ corresponds to a point  in $ \Delta_{m} = \{\vec p \in \R_+^{m+1} : \sum_{j=1}^{m+1} p_j =1\}$, which 
  is polytope associated to  $\CP^{m}$.  Given $k \in \Z_{\geq 1}$, let $\vec \alpha \in \Z_+^{m+1}$ and let $k = |\vec \alpha|$. 
  If $\vec x \in \R^{m+1}$ let $\vec x^{\vec \alpha} = \prod_{j=1}^{m+1} x_1^{\alpha_1} \cdots x_{m+1}^{\alpha_{m+1}}$. 
 A random vector $\vec X = (X_1, \dots, X_{m+1})$ has the multinomial distributions with parameters $k$ and $\vec p$
 if $\rm{Prob} [\vec X = \vec \alpha] = {k \choose \vec \alpha} \vec p^{\vec \alpha}$ where ${k \choose \vec \alpha} = \frac{k!}{\vec \alpha!}$.

It is proved in \cite[Theorem 1]{CG12} and \cite{Mat78} that the multinomial distributions with parameters $k$ and $\vec p= (p_1, \dots, p_{m+1})$ 
has the asymptotic form,
$$H(\mu_p^{*k}) = \half \log ((2 \pi k e)^{m } p_1 \cdots p_{m+1}) + \frac{1}{12 k}\left(3(m+1) - 2 -\sum_{j=1}^{m+1} \frac{1}{p_j} \right) + O(\frac{1}{k^2}). $$

\begin{rem} Since $\sum_{j=1}^{m+1} p_j =1$, there are only $m$ independent $p_j$. In the formula of Theorem \ref{H},  the
parameter $m + 1$ in the   multinomial case corresponds to $\CP^{m}$, so the coefficients of $\log k$ agree.\end{rem}

Aside from asymptotic entropies of multinomial distributions, there  exist  few general results on asymptotic entropies of convolution
powers $\mu^{*k}$.
  Asymptotics of entropies to several orders for  certain classes of discrete distributions  as $k \to \infty$ were obtained in
  \cite{K98, JSz99}.
In the case of sums of i.i.d. real-valued random variables, i.e.  convolution powers of probability measures on $\R$, Dyachkov proved
in \cite[Theorem 2]{D96} that
$$H(\mu^{*k})
\simeq \half( \log k ) + \half \log (2 \pi e \sigma^2) + o(1).$$

In view of the resemblence of the entropy asymptotics of the toric \kahler probability measures $\mu_k^z$ to convolution powers, it is natural
to characteristic the toric Hermitian line bundles $(L, h) \to (M, \omega)$ for which $\mu_k^z$ is a sequence of convolution powers. 

\begin{theo} \label{INVCONV} The sequence $\{\mu_k^z\}_{k=1}^{\infty} $ is a sequence of convolution powers for all $z$ if and only if
$\rm{Hilb}_k(h)$ is {\it balanced} for all $k$, i.e. the density of states  $\Pi_{h^k}(z) = D_k$ is constant for all $k$. Hence, $\omega$ is a \kahler
metric of constant scalar curvature; \bigskip


\end{theo}
To prove Theorem \ref{INVCONV} we first prove a result about balanced metrics on any \kahler manifold which seems of
independent interest.

     \begin{prop}\label{EQUIV}  For any \kahler manifold $(M, \omega, J)$, the following are equivalent: \begin{enumerate}
\item $\rm{Hilb}_k(h)$ is {\it balanced} for all $k$, i.e. the density of states  $\Pi_{h^k}(z) = C_k$ is constant for all $k$.  \bigskip

\item $\Pi_{h^k}(z,w) = A_k  [\Pi_{h^1}(z,w)]^k$, where 
$$A_k = \frac{\dim H^0(M, L^k)}{(2\pi)^m Vol(P)}    \left( \frac{(2 \pi)^m  Vol(P)} { \dim H^0(M, L)}\right)^k.$$\bigskip
\end{enumerate}

\end{prop}

In the case of a toric \kahler manifold, $$A_k = \left(\frac{\#\{\alpha \in k \overline{P} \cap \Z^m\}}{(2\pi)^m Vol(P)} \right)    \left( \frac{(2 \pi)^m  Vol(P)} { \# \{\alpha \in \overline{P} \cap \Z^m\}}\right)^k.$$

We refer to \cite{D02} for background and results on balanced and constant scalar curvature metrics on toric \kahler manifolds.


\subsection{Ricci curvature and measures of maximal entropy}

     The entropy $H(\mu)$ of a discrete probability 
  measure $\mu$ is a measure of the degree to which $\mu$ is uniform. The larger the entropy, the more uniform the measure, so that 
the measure of maximal entropy in a given family of probability measures is the most uniform measure. This measure of maximal entropy is
often considered the most important. Hence it is natural to ask for which $z$ does $\mu_k^z$ have maximal entropy in the family $\mu_k^z$, at
least asymptotically as $k \to \infty$. For instance, in the case of binomial measures $\mu_p^{*k}$, $p = \half$.

Locating the point $\mu(z) = x$ where $\mu_k^z$ has asymptotically maximal entropy is related to the Ricci curvature  of $(M,\omega)$.
We recall that the Ricci curvature
of the \kahler metric $\omega_{\varphi}$ is given by $\rm{Ric}(\omega) = - i \ddbar \log \det (g_{i \bar{j}})$, i.e. $\rm{Ric}_{k \ell}= - \frac{\partial^2}{\partial z_k \partial \bar{z}_{\ell}} (\log \det g_{i \bar{j}})$ where
$\omega = \frac{i}{2} g_{i \bar{j}} dz^i \wedge d\bar{z}^j$.   In \cite{Ab98} it is shown that in the toric case,
\begin{equation} \label{RICCI} \rm{Ric} = - \half dd^c \log \det H  = - \half \sum_{i, j, k}^m H_{ij, jk} dx_k \wedge d \theta_i, \end{equation}
Thus, the  Ricci potential is the function $- L(x)$ \eqref{LDEF}.

Due to the inverse relation of $  i \ddbar \phi$ and $\nabla^2 u$,  points where the Ricci potential is maximal are points where \eqref{LDEF} is minimal.
In the simplest case of the Fubini-Study symplectic potential on $\CP^1$, in a standard gauge the symplectic potential satisfies,
 $\log u_{FS}''(x) = - \log x (1-x)$,  and $\frac{d^2}{dx^2} \log u_{FS}''(x) = x^{-2} + (1 - x)^{-2}$. The unique minimum point of $\log u''_{FS}$ occurs at $x= \half$.
 In the case of multinomial distributions and Fubini-Study potentials in higher dimensions, the maximum occurs at the center of mass of the simplex. These are model cases of toric Fano \kahler-Einstein manifolds.
 It turns out that related statements are true for  compact toric   \kahler manifolds with positive Ricci curvature.  
 We recall that  $\rm{Ric} (\omega)$ represents the first Chern class $c_1(M)$ and $\rm{Ric}  > 0$ implies that $(M, \omega)$ is a toric Fano manifold.
That is,  if $\rm{Ric}(\omega) > 0$, then $\omega$ is a positively curved metric on the anti-canonical bundle $-K_X$, hence $-K_X$ is ample. A toric Fano 
manifold has a distinguished center, namely the center of mass of polytope.  We refer to \cite{D08} for background and results on toric
 Fano \kahler manifolds and their preferred centers.

\begin{theo}\label{MAX}  For fixed $(L, h, M, \omega)$, the points $x = \mu(z)$ for which  the measures $\mu_k^z$ have asymptotically maximal
entropy as $k \to \infty$   occur at the minimum points of $L(x)$ \eqref{LDEF}. If $(M, \omega)$ is Fano and $\rm{Ric}(\omega)$ is positive, then there is a unique minimum.
In the  \kahler-Einstein Fano case, where $\Ric(\omega) = a \omega$, the point of maximal entropy is the center of mass of $P$ (which equals
$0$ if $P$ is put in  the  form of  \cite{M87}.)
 \end{theo}

For instance, in the case of Fubini-Study metrics on $\CP^m$, the open orbit \kahler potential is $\log (1 + |w|^2), w \in \C^m$, and 
$-\log \det \nabla^2 \varphi (\rho)  = (m+1) \varphi(\rho) - \rho. $ The unique point of maximal entropy is given by $ \frac{e^{\rho}}{1 + e^{\rho}} = \frac{1}{m} (1, \dots, 1)$.
In the gauge  of Mabuchi \cite{M87}, where the polytope is translated by  $-\frac{1}{m} (1, \dots, 1)$, the unique point is $0$ (see Section \ref{GAUGESECT} for
background on gauges).


\begin{rem} The Mabuchi functional $\mcal(\omega)$  on \kahler metrics involves the relative entropy of $\omega_{\phi}^m$ and of a background volume form. 
As shown in \cite[Proposition 3.2.8]{D02}, it  is given on a toric \kahler manifold  by $\mcal(\omega) = (2 \pi)^n \fcal_a(u)$ where,
$$\fcal_a(u) = \int_P L(x) dx + \int_{\partial P} u d \sigma - a \int_P u dx, $$
where $a = \frac{\rm{Vol}(\partial P, d \sigma)}{\rm{Vol}(P)}$ where $d\sigma$ is Euclidean surface measure.
\end{rem}

\subsection{Differential entropy of the Gaussian measure $\gamma_{h^k}$}

There is a second (and much simpler) problem regarding entropies of probability measures on a toric \kahler manifold, or indeed on any
 polarized \kahler manifold. 
    Associated to any Hermitian metric $h$ on $L$ is a sequence  $\{\rm{Hilb}_k(h)\}_{k=1}^{\infty}  $ of Hermitian   inner products on $H^0(M, L^k)$.     In turn the inner product induces a Gaussian measure $\gamma_{h_k}$ on $H^0(M, L^k)$.  If we fix a background metric
    $h_0$, or corresponding inner product $G_0$, then the inner product $\rm{Hilb}_k$ is represented by a positive Hermitian matrix $P$ and the Gaussian measure $\gamma_k^h$
    is represented by $\det P e^{- \langle P^{-1} X, X \rangle}$ on $\C^{N_k}$ where $N_k = \dim_{\C} H^0(M, L^k)$,

When a probability measure  $\mu$ on $\R^n$ has a  density $f$ relative to Lebesgue measure $dx$, its {\it differential entropy} is
defined by 
$$H(f dx) = - \int_{\R^n} f(x) \log f(x) dx. $$
It is well-known that if $f(x) = N(\mu, \sigma) = \frac{1}{\sqrt{2 \pi }\sigma} \exp \left(- \frac{(x - \mu)^2}{2 \sigma^2} \right)$ is a Gaussian, then,
$$h(f dx) = \ln (\sigma \sqrt{2 \pi e}). $$

We now calculate the differential entropy of the Gaussian measures $\gamma_k^h$. 

\begin{prop} \label{HILBH} Let $(L, h, M, \omega)$ be any polarized \kahler manifold, and let $\gamma_k^h$ be the associated
Gaussian measure on $H^0(M, L^k)$. Then $H(\gamma_k^h) = - \log \det \rm{Hilb}_k(h) $. The Hermitian metric $h$ for which $H(\gamma_k^h)$ has maximal entropy is the balanced metric. \end{prop}

\subsection{Further problems  on the sequence of toric measures}  Although entropy has a natural interpretation for a single probability measure (its degree
of uniformity), it plays a more essential role in the dynamics of Markov chains (the Shannon-Breiman-MacMillan theorem; see \cite{KS}).

A well-known Markov chain is the so-called Wright-Fisher Markov chain: Let $X_n$ $(n \geq 0)$ be the Markov chain with state space $\{0, \dots, N\}$ and with
transition probabilities,
$$p_{ij} : = P(X_{n+1} ) = j | X_n = i) =  {N \choose j} (\frac{i}{N})^j (1 - \frac{i}{N})^{N-j},\;\;(i,j = 0, \dots, N). $$
A  more general Wright-Fisher Markov chain is to define
$$p_{ij} = {N \choose j} \psi_i^j (1 - \psi_i)^{N -j}, $$
where $\psi_i$ are other weights of the lattice points. There is a straightforward generalization to toric \kahler manifolds by defining the right stochastic matrix 
$$P^{(N)}_{\alpha \beta} = \frac{|s_{\alpha}(\beta)|^2_{h^N}}{Q_{h^N}(\alpha) \Pi_{h^N}(\beta, \beta)}. $$
In fact, we could take any orthonormal basis $\{s_{N, j}\}$ and any points $\{z_{N, k}\}_{k=1}^{d_N}$ and form the 
Markov chain with 
$$P^{(N)}_{j, k} : = \frac{|s_{N, j}(z_{N, k})|^2_{h^N}}{\Pi_{h^N}(z_{N, k}, z_{N, k})}. $$
It might be of interest  to determine  the asymptotic entropy of this Markov chain, which is closely related to the measures $\mu_k^z$.



\begin{rem}
The article \cite{DK}  also considers entropy in the context of Bergman kernels, but does not seem to overlap this article. It is devoted to the simpler question of when the density of states $\Pi_{h^k}(z)$ has maximal entropy (it is 
evidently the balanced metric) and its applications to black hole physics. \end{rem}

\subsection{Acknowledgements} Thanks to Peng Zhou for many helpful conversations, particularly about  gauge freedom for toric \kahler manifolds. 
Thanks also to the referee for pointing out some errors and omissions in an earlier version, which led the authors to  add Theorem \ref{MAX}.

\section{\label{BACKGROUND} Background on toric varieties}

We employ the same notation and terminology as in
\cite{SoZ12,ZZ18}. We recall that a toric \kahler manifold is a \kahler manifold
$(M, J, \omega)$ on which the complex torus $(\C^*)^m$ acts
holomorphically with an open orbit $M^o$.  We
 choose a basepoint $z_0$ on the orbit open and
identify  $M^o \equiv (\C^{*})^{m} \{z_0\}$. The underlying real torus is
denoted $\T$ so that $(\C^*)^m = \T \times \R_+^m$, which we write
in coordinates as $z = e^{\rho/2 + i \theta}$ in a multi-index
notation. Thus, $|z|^2 = e^{\rho}$. We often express the \kahler potential in  $\rho$ coordinates.

We assume that $M$ is a smooth projective toric \kahler manifold, hence that  $P$ is a Delzant
polytope, i.e. that  $P$ is
 defined by a set of linear inequalities
\begin{equation} \label{lrdef} \ell_r(x): =\langle x, v_r\rangle-\alpha_r \geq 0, ~~~r=1, ..., d,  \end{equation}
where $v_r$ is a primitive element of the lattice and
inward-pointing normal to the $r$-th $(n-1)$-dimensional face of
$P$. We denote by $P^o$ the interior of $P$ and by $\partial P$
its boundary; $P = P^o \cup
\partial P$. 

\subsection{Toric line bundles and their powers}
We consider powers $L^k \to M$ of an ample toric line bundle $L \to M$ with $k \in \Z$. A model case is that of powers $\ocal(k) \to \CP^m$ of the
dual line bundle $\ocal(1) \to \CP^m$ of the hyperplane line bundle $\ocal(-1) \to \CP^m$.  Positive powers $\ocal(k)$ are ample line bundles and
the holomorphic sections correspond to monomials $z^{\alpha}$ on $\C^{m+1}$ with $|\alpha| = k$.
Negative powers have no holomorphic sections.  

The canonical line bundle $\kcal$ is the top exterior power of the holomorphic cotangent bundle; its sections are smooth $(m,0)$ volume forms. 
Theorem \ref{MAX} concerns Fano toric \kahler manifolds, namely manifolds with ample anti-canonical line bundle (hence, negative canonical line bundle). 
A model example is $\CP^m$, for which $\kcal = \ocal(-(m+1))$.

We need to linearize  (or quantize) the torus action so that it acts
 on $H^0(M, L^k)$. It is sufficent to lift the action to $L^*$.  The lifting procedure is described in \cite[Lemma 1.1]{ZZ19} for single Hamiltonians, and
 essentially the same procedure works to define lifts of the commuting Hamiltonians of a torus action.  We equip $L$ with a toric Hermitian metric $h$ whose
curvature $(1,1)$-form  $\omega$. The  Hermitian metric $h$ on $L$ induces a Chern connection on the $S^1$ bundle $X_h = \partial D_h^* \to M$ where
$D_h^* \subset L^*$ is the unit co-disk bundle with respect to $h$. We then lift the Hamilton vector fields $\xi_{H_j}$ generating the torus action on $M$
to contact vector fields $\hat{\xi}_{H_j} = \xi_{H_j}^h - 2 \pi H_j R$ on $X_h$ where $\hat{\xi}_H^h$ denotes the horizontal lift of $\xi_H$ and where $R$ is
the Reeb vector field generating rotations in the fibers of $X_h \to M$. The vertical and horizontal parts commute, and if $\xi_{H_j}$ commute then
their horizontal lifts commute. We may choose generators so each $\xi_{H_j}$ generates a $2 \pi $-periodic flow (such Hamiltonians are known
as action variables). It is verified in \cite[Lemma 2.6]{ZZ19} that
$\hat{\xi}_{H_j}$ also generates a $2 \pi $-periodic flow. Together with $R$, one has a ${\mathbb T}^{m+1}$ action on $X_h$. The Hamiltonians $H_j$ are
not uniquely defined because one may add a constant $c_j$ to each without changing $\xi_{H_j}$. However, in order that the lifts $\hat{\xi}_{H_j}$
generate periodic flows, it is only possible to add a lattice point $\vec k \in \Z^m$ to the vector $(H_1, \dots, H_m)$ of Hamiltonians. Thus, the possible lifts
form a $\Z^m$-family.

For each choice of lift and each power $L^k$ of the ample toric line bundle, there exists a unique (up to scalars) torus-invariant section, whose restriction to the open orbit
we denote by  $e_{L^k} $.  See \cite{Fu} or \cite[Section 5]{GS82} (which
treats general compact Lie groups). In the case of $\ocal(k) \to \CP^m$ it corresponds to the lattice point $\alpha = 0$.
 For $k = m+1$, the invariant section may be viewed as the multivector $(z_1 \frac{\partial}{\partial z_1}) \wedge \cdots \wedge (z_m \frac{\partial}{\partial z_m})$
dual to the meromorphic invariant volume form $\frac{dz_1}{z_1} \wedge  \cdots  \wedge \frac{dz_m}{z_m},$
which has an order 1 pole at each boundary divisor. 

 A natural basis of the
space of holomorphic sections $H^0(M, L^k)$ associated to the
$k$th power of an ample toric holomorphic line bundle $L \to M$ is the basis of equivariant sections, i.e. holomorphic eigensections of the linearized torus action.
They  are defined in the open orbit  by the  monomials $z^{\alpha}$
where $\alpha$ is a lattice point in the $k$th dilate of the
polytope, $\alpha \in k P \cap \Z^m.$ To be more precise, on the open orbit  $s_{\alpha}(z) =
z^{\alpha} e_{L^k}$. 
Let $I_k \subset \Z^m$ be the subset consisting of the weights in $H^0(M, L^k)$ under the action of $(\C^*)^m$, and let $P_k$ be the convex hull of $I_k$. Then  $P_k = k P'$ for a fixed convex polytope $P'$. We denote the dimension of $H^0(M, L^k)$ by $N_k$.
For background, see \cite{Fu}.

\subsection{\label{MONSECT}Inner products and norms of monomials}

We equip $L$ with a toric Hermitian metric $h$ whose
curvature $(1,1)$-form may be expressed in terms of a local holomorphic frame $e_L$ by  $\omega = i \ddbar \log \|e_L\|_{h}^2$.  
Any  hermitian metric $h$ on $L$ induces inner products
$\Hilb_k(h)$ on $H^0(M, L^k)$, defined by
\begin{equation} \label{HILB} \langle s_1, s_2 \rangle_{\text{Hilb}_k(h)} =
\int_M (s_1(z), s_2(z))_{h^k} \frac{\omega_h^m}{m!}.
\end{equation} The equivariant sections (monomials) are orthogonal with respect to any
such toric inner product. We often express the norm in terms of a local
\kahler potential, $\|e_L\|_{h}^2 = e^{- \varphi}$, so that
$|s_{\alpha}(z)|_{h^k}^2 = |z^{\alpha}|^2 e^{- k \varphi (z)}$ for
$s_{\alpha} \in H^0(M, L^k)$.
The $L^2$ norm-square of
 $s_{\alpha}$ with respect to the natural inner product $\text{Hilb}_k(h)$  induced by the Hermitian metric on $H^0(M, L^k)$ is given by,
\begin{equation} \label{QFORM} Q_{h^k}(\alpha) =  \|s_{\alpha}\|_{h^k }^2 =  \int_{\C^m} |z^{\alpha}|^2 e^{-
k \varphi(z)} dV_{\varphi}(z), \end{equation}
  Here,  $dV_{\varphi} = (i
\ddbar \varphi)^m/ m!$.


\subsection{\label{BACKGROUNDOO} \kahler potential, moment map  and sympletic potential} 
An open-orbit  \kahler potential is a real-valued  torus invariant function $\varphi$ such that $i \ddbar \varphi = \omega$ on the open orbit. 
This potential is only defined up to an additive pluri-harmonic torus invariant function (i.e. an affine function).  Recall that the log coordinates $(\rho, \theta)$ on $M^o \cong (\C^*)^m$ are defined by setting $z_i = e^{\rho_i/2 + \sqrt{-1} \theta_i}$. 
Since the Kahler potential $\varphi$ is $\T$-invariant, $\varphi(z)$ only depends on the $\rho$ variables, hence we may write it as $\varphi(\rho)$ ($e^{\rho} = |z|^2$).
The associated moment map $\mu_h$ is defined as the gradient of the \kahler potential,
\[\mu_h: \R^m_\rho \to P \subset  \R^m_x, \quad \rho \mapsto  \nabla \phi(\rho) . \]
The polytope $P$ is the image of the moment map.
The  moment map $\mu_h: M \to \R^m$ is only well-defined up to an additive constant vector; hence $P$ is only defined up to translation without
further normalization. As this indicates, there are a number of implicit choices in the definition of the \kahler potential and moment map.

Given the norming constants \eqref{QFORM},  a standard definition of   the open-orbit  \kahler potential is,
\begin{equation}\label{OOKP} \varphi(\rho) := \log \left( \sum_{\alpha \in P \cap \Z^m} \frac{e^{\langle \alpha, \rho \rangle}}{Q_{h}(\alpha) } \right). \end{equation}
The sum is essentially the density of states (the value of the Bergman kernel on the diagonal) (see Section \ref{SZEGOSECT}).  More precisely, it is the modulus
square of the coefficient of the Bergman kernel relative to the invariant frame, i.e. the 
ratio of the density of states and the norm square $|e_L|^2$ of the invariant section. These definitions depend on the choice of linearization of the torus action. 
If the invariant section corresponds to the lattice point $\alpha_0$, then the exponent in $e^{\langle \alpha, \rho \rangle}$ in \eqref{OOKP}  is translated to
$e^{\langle \alpha - \alpha_0, \rho \rangle}$.

We now consider the {\it symplectic potential} $u_0$ associated to
$\phip$, defined as the  Legendre transform  of $\varphi$ on $\R^m$:
\begin{equation} \label{SYMPOTDEF} u_{\varphi}(x) = \varphi^*(x) =  \lcal \varphi(x): = \sup_{\rho \in \R^m}
(\langle x, \rho \rangle - \varphi(e^{\rho/2 + i \theta})).
\end{equation}  It is a function on $P$, or in  invariant terms it is a function on   $
Lie(\T)^* \simeq Lie(\R_+^m)^*$. In general, the Legendre
transform of a function on a vector space $V$ is a function on the
dual space $V^*$.  

Thus, 
\begin{equation} \label{SYMPOTDEF2a} u_{\varphi}(x) = \langle x, \rho_x \rangle -
\varphi(\rho_x), \;\;\; e^{\rho_x/2 } = \mu_{\varphi}^{-1}(x) \iff
\rho_x = 2 \log   \mu_{\varphi}^{-1}(x)
\end{equation}  on $P$.  The
gradient $\nabla_x u_{\varphi}$ is an inverse to
$\mu_{\omega_{\varphi}}$ on $M_{\R}$ on the open orbit, or
equivalently on  $P$, in the sense that $\nabla u_{\varphi}
(\mu_{\omega_{\varphi}}(z)) = z$ as long as $\mu_{\omega_{\varphi}}(z)
\notin \partial P$.

The symplectic potential has canonical
logarithmic singularities on $\partial P$. According to \cite{Gu94} and to
 \cite{D02} ( Proposition 3.1.7),
 \begin{equation} \label{CANSYMPOT}
u_0(x) = \sum_k \ell_k(x) \log \ell_k(x) + f_0 \end{equation}
where $f_{0} \in C^{\infty}(\bar{P})$. The Guillemin canonical metric is the special case where $f_0 =0$.

\subsection{\label{GAUGESECT} Gauges and examples}

As mentioned above, \kahler potentials and the corresponding symplectic potentials are not unique, and we refer to a choice of one potential
as a choice of gauge. The gauge symmetries of the pairs $(\varphi, u)$ are as follows. We assume that the arguments $\rho$ of $\varphi$ and
$x$ of $u$ are related by $x = \mu_h(e^{\rho/2})$.
\begin{itemize}
\item $\varphi \to \varphi  + c, u \to u -c$ for any $c \in \R$; \bigskip

\item $\varphi(\rho)  \to \varphi(\rho) + \vec b \cdot \rho,  x  \to x + \vec b =  \nabla ( \varphi(\rho) + \vec b \cdot \rho)\ \; \vec b \in \R^m$; \bigskip

\item $u(x) \to u(x) + \vec k \cdot x, \rho \to \rho + \vec k$. \bigskip

\item If we only choose \kahler potentials \eqref{OOKP} corresponding to Bergman kernels and torus-invariant sections, then $\vec b \in \Z^m$.

\end{itemize}

Let us illustrate the definitions and ambiguities with the \kahler potential, moment map and symplectic potential for the Fubini-Study metric of $\CP^m$.
In the case of $\ocal(1) \to \CP^1$, a standard choice for the 
Fubini-Study \kahler potential is $\varphi(z) =\log (1 + |z|^2) =
\log (1 + e^{\rho}) = \varphi (e^{\rho}) = \varphi(\rho)$ (with a little abuse of notation) and the moment map is $\mu_{FS}(\rho) = \frac{e^{\rho}}{1 + e^{ \rho}}.$
As in the introduction, the polytope is $[0,1]$ and  the correspondinng  symplectic  potential is $u_{FS}(x) = x \log x + (1 -x) \log (1 - x)$. 

However, other gauge choices are possible. If we allow all Chern classes and intervals, a  second choice of  \kahler potential is $\varphi_{r} (\rho) = r^2 \log \cosh \rho$. The corresonding 
 symplectic potential is,
$$u_{r}(x) = \half \left((r^2 + x) \log (r^2+ x) + (r^2 - x) \log (r^2 - x)\right),$$  where the parameter $r$ determines the radius of the corresponding round $S^2$, and
corresponding   polytope is   $[- r^2, r^2]$.
 The radius parametrizes  the cohomology class of the \kahler form. which is a translate of the centered polytope $[-\half, \half]$  with
$r^2 = \half$.

As a higher dimensional example, consider
the canonical bundle $\ocal(-(m+1)) \to \CP^m$ with  the Fubini-Study metric. The image of the moment map for this potential is $(m+1)$ times the standard unit simplex in $\R^m$. The volume form is given by $\omega_{FS}^m = (\frac{i}{2})^m \frac{\prod_{i=1}^m dz_i
\wedge d\bar{z}_i}{(1 + |z|^2)^{m+1}}$ and $- \log (1 + |z|^2)^{m+1}$ is a \kahler potential for the Ricci form. On the other hand, one may express the
volume form in terms of the invariant $(m,0)$ form $\prod_{i=1}^m \frac{dz_i}{z_i}$.  That changes the \kahler potential by $- \sum_j \log |z_j|$ and translates
the simplex by $-(1,1, \cdots, 1)$ so that it is centered at $0$.

\subsection{Fano toric \kahler manifolds}
Theorem \ref{MAX} pertains to Fano toric \kahler manifolds, namely those with positive anti-canonical bundle. The  polytope of a Fano \kahler manifold  has a preferred center
$x_0 \in P$ such that $\ell_j(x_0) = 1$ for all $j$. As explained in  \cite[Section 3.1]{D08}, this  follows because the wedge product of the
vector fields generating the torus  action is a meromorphic m-form on $M$  with a simple
pole along each of the divisors corresponding to the faces. Its  inverse is a
section of the anti-canonical bundle. The centre $x_0$
is also the centre of mass of $\partial P$ with its induced surface measure.
The center is $0$ if and only if the metric is \kahler-Einstein \cite{M87, WZ04}; equivalently, vanishing of the Futaki invariant is equivalent to the fact that
the preferred center is the center of mass.



\subsection{\label{SZEGOSECT} \szego (or, Bergman) kernels}

 The \szego (or Bergman) kernels of a positive Hermitian line
bundle $(L, h) \to (M, \omega)$  are the kernels of the orthogonal
projections $\Pi_{h^k}: L^2(M, L^k) \to H^0(M, L^k)$ onto the
spaces of holomorphic sections with respect to the inner product
$\Hilb_k(h)$,
\begin{equation}\label{Pik} \Pi_{h^k} s(z) = \int_M \Pi_{h^k}(z,w) \cdot s(w)
\frac{\omega_h^m}{m!}, \end{equation} where the $\cdot$ denotes
the $h$-hermitian inner product at $w$.
 In terms of a local frame  $e$  for $L \to M$ over an
open set $U \subset M$,  we may write sections as $s = f e$. If
$\{s^k_j=f_j e_L^{\otimes k}:j=1,\dots,N_k\}$ is  an orthonormal
basis for $H^0(M,L^k)$, then  the \szego kernel can be written in
the form
\begin{equation}\label{szego}  \Pi_{h^k}(z, w): = F_{h^k}
(z, w)\,e_L^{\otimes k}(z) \otimes\overline {e_L^{\otimes
k}(w)}\,,\end{equation} where
\begin{equation}\label{FN}F_{h^k}(z, w)=
\sum_{j=1}^{N_k}f_j(z) \overline{f_j(w)}\;, ~~~N_k = \dim H^0(M, L^k).\end{equation}
We also introduce the local kernel $B_k(z,w)$, defined with respect to the unitary frame: 
\begin{equation} \label{SzK} 
\Pi_{h^k}(z,w) = B_k(z,w) \cdot \frac{e_L^k(z)}{\|e^k_L(z)\|_h} \ot \overline{\frac{e_L^k(w)}{\|e^k_L(w)\|_h}} 
\end{equation}
The  {\it density of states} $\Pi_{h^k}(z)$  is the contraction of $\Pi_{k}(z,w)$ with the hermitian metric on the diagonal, 
\[\Pi_{h^k}(z): =   \sum_{i=0}^{N_k}
\|s^k_i(z)\|_{h_k}^2= F_{h^k}
(z, z)\,\left |e(z) \right |^{2k}_{h} =\ B_k(z,z), \]
where in the first equality we record a standard abuse of notation in which the diagonal of the
\szego kernel is identified with its contraction. 
On the diagonal, we have the following asymptotic expansion the density of states,
\begin{equation} \label{TYZ}  \Pi_{h^k}(z) =  k^m (a_0 + a_1 S(z) k^{m-1} + a_2(z) k^{m-2} +
\cdots )  \end{equation} where  $S(z)$ is the
scalar curvature of $\omega$. The leading order term $a_0 =1$ (see Section \ref{AG}) if $\Pi_{h^k}(z,w)$ is the \szego kernel relative to the volume form $dV = \frac{\omega^m}{m!}$.

 \subsection{Bergman kernels for a toric variety} In the
case of a toric variety, we have
\begin{equation}\label{FNa}F_{h^k}(z, w)=
\sum_{\alpha \in k P \cap \Z^m}  \frac{z^{\alpha}
\bar{w}^{\alpha}}{Q_{h^k}(\alpha)} \;,\end{equation}   where $Q_{h^k}(\alpha)$ is defined in \eqref{QFORM}. If we sift out
the $\alpha$th term of $\Pi_{h^k}$  by means of Fourier analysis
on $\T$, we obtain \eqref{PHK} .

Let $\wt \varphi(z,w)$ denote the almost extension of $\varphi(z)$ from the diagonal, that is $\wt \varphi$ satisfies the  condition $\dbar^k_z\wt   \varphi(z,w)|_{z=w} = \d^k_w \wt \varphi(z,w)|_{z=w} = 0$ for all $k \in \N$ and $\wt \varphi(z,w)|_{z=w} = \varphi(z)$. The $\T$ action is by holomorphic
isometries of $(M, \omega)$ and therefore \begin{equation} \label{EQUI} \wt \varphi(\Phi^{\vec t} z,\Phi^{\vec t} w) = \wt \varphi(z,w).  \end{equation}

The \szego kernel \eqref{SzK} admits a parametrix with complex phase
$\wt \varphi$.  In the case of a toric \kahler manifold,
it takes the following simple form \cite{STZ03}.

\begin{prop}\label{SZKTV}  For any hermitian  toric positive line bundle over a toric
variety, the \szego kernel for the metrics $h_{\varphi}^N$ have the
asymptotic expansions in a local frame on $M$,
$$B_{h^k}(z, w) \sim e^{k \left(\wt \varphi(z, w) - \frac{1}{2} (\varphi(z)
+ \varphi(w)) \right) } A_k(z,w) \;\; \mbox{mod} \; k^{- \infty},
$$ where $A_k(z,w) \sim k^m \left(1 + \frac{a_1(z,w)}{k} + \cdots\right) $ is a semi-classical symbol of order $m$ and where the phase satisfies
\eqref{EQUI}. \end{prop}

\subsection{\label{AG} Facts from algebraic geometry}

If $L^k$ is very ample (i.e. the vanishing theorem holds), then
$N_k: = \dim H^0(M, L^k)$ satisfies $$N_k +1 = \chi(L^k) = \int_M e^{k c_1(L)} Td(M) = a_0 k^m + a_1 k^{m-1} + \cdots + a_m,$$
where $$a_0 = \frac{1}{m!} \int_M c_1(L)^m, \;\; a_1 = \frac{1}{(2 (m-1)!} \int_M c_1(L)^{m-1} c_1(M). $$ 
Also, $[\omega] = 2 \pi c_1(L)$. Also
$$\int_M \Pi_{h^k}(z,z) dV_{\omega} = \dim H^0(M, L^k) = a_0 k^m + a_1 k^{m-1} + \cdots, $$
with $$a_0 = \rm{Vol}_{\omega}(M) = \int_M  \frac{\omega^m}{m!}, a_1 = \frac{1}{2 \pi} \int_M S(\omega) dV_{\omega}. $$
Here, $dV_{\omega} = \frac{\omega^m}{m!} = d\mu$.

\subsection{Asymptotic results on $\qcal_k(\alpha)$ and  $\pcal(\alpha, z)$}

In \cite[(23)]{SoZ12}, the  norming constants are expressed  in terms
 of the symplectic potential:
\begin{equation} \label{SPNORM} \QQ_{h^k}(\alpha) = \int_P e^{ k
(u_{0}(x) + \langle \frac{\alpha}{k} - x, \nabla u_{0}(x) \rangle}
dx. \end{equation} For  interior $\alpha$, and $\alpha_k$
with $|\alpha - \alpha_k| = O(\frac{1}{k})$,
\begin{equation} \label{QQ} \QQ_{h^k}(\alpha_k) \sim k^{-m/2} e^{ k
u_0
(\alpha)}, \end{equation} and for all  $\alpha$ and $\alpha_k$
with $|\alpha - \alpha_k| = O(\frac{1}{k})$,
\begin{equation} \label{QQa} \frac{1}{k} \log \QQ_{h^k}(\alpha_k) =
u_0
(\alpha) + O(\frac{\log k}{k}). \end{equation}

The weights $ \pcal_{h^k}(\alpha, z)$ \eqref{PHK} of the dilate
$\mu_k^{z,1}$ admit pointwise asymptotic expansions. 
The following is \cite[Lemma 5.1]{ZZ18}.

\begin{lem} \label{pcalLem} $\pcal_{h^k}(\alpha, z) = k^{m/2} (2\pi)^{-m/2}|\det  {\Hess}(u_{\varphi}(\frac{\alpha}{k})|^{\half} e^{- k I^z(\frac{\alpha}{k})} (1+O(1/k)), $
where $O(1/k)$ is uniform in $z, \alpha$. 
If $|\mu(z) - \frac{\alpha}{k} | = O(\frac{1}{k})$, then $$\pcal_{h^k}(\alpha, z) = k^{m/2} (2\pi)^{-m/2}|\det  {\Hess}(u_{\varphi}(\mu_h(z))|^{\half} e^{- k I^z(\frac{\alpha}{k})} (1+O(1/k)), $$
where $O(1/k)$ is uniform in $z, \alpha$.

\end{lem}



\subsection{Probabilistic results}

  In \cite{SoZ12} the following is proved:


\begin{prop} \label{COMPAREPIT2} Let $(M, L, h, \omega)$ be a polarized toric Hermitian line bundle. Then the means, resp. variances of $\mu_k^z$ are given respectively by,

\begin{enumerate}

\item $m_k(z)=   \mu_h(z) + O(k^{-1});$

\item $\Sigma_k(z) =  k^{-1}  {\Hess}\; \varphi +
 O(k^{-2}) $.

\end{enumerate}
\end{prop}
Moreover, the measures $\mu_k^z$ satisfy a weak law of large numbers; see \eqref{WEAK}.

Let $h = e^{-\varphi}$ be a toric Hermitian metric on $L$. Recall that 
the {\it symplectic potential} $u_{\varphi}$ associated to $\varphi$
 is its Legendre transform: for $x \in P$ there
is a unique $\rho(x)$ such that $\mu_{\varphi}(e^{\rho(x)/2}) =d\varphi (\rho(x))= x$. If $z = e^{\rho/2 + i \theta}$ then we write
$\rho_z = \rho = \log |z|^2$. Then the Legendre transform is defined to
be the convex function
\begin{equation} \label{SYMPOTDEF2} u_{\varphi}(x) = \langle x,  \rho(x) \rangle -
\varphi(\rho(x)).
\end{equation} 
Also define
  \begin{equation} \label{Ikzdef} I^z(x) = u_{\varphi} (x) -
\langle x, \rho_z \rangle + \varphi (\rho_z).\end{equation} 
Then $I^z(x)$ is a convex function on $P$ with a minimum of value $0$ at $x = \mu_h(z)$ and with Hessian that of $u_{\varphi}$.

\subsection{\label{LDSECT} Large deviations}

In  \cite{SoZ12} it is proved that the measures $\mu_k^z$ satisfy a large deviations principle with speed $k$ and a rate
function $I^z$. 
The rate functions $I^z$ for $\{d\mu_k^z\}$ depend on whether $z$ lies
 in the open orbit $M^o$ of $M$ or on the divisor at infinity $\dcal$. The following is proved in \cite{SoZ12}.

\begin{theo}\label{LDINTRO}  For any  $z \in M$,  the probability measures
$\mu_k^z$  satisfy a uniform Laplace  large deviations principle
with rate $k$ and with convex rate functions $I^z \geq 0$ on $P$
defined as follows:

\begin{itemize}

\item If $z \in M^0$, the open orbit, then  $I^z(x) = u_0(x) -
\langle x, \log |z| \rangle + \phip (z),$  where $\phip$ is the
canonical \kahler potential of the open orbit and $u_0$ is its
Legendre transform, the  symplectic potential;

\item When $z \in \mu_0^{-1}(F)$ for some face $F$ of $\partial
P$, then $I^z(x)$ restricted to $x \in F$ is given by
 $I^z(x) =  u_F(x) - \langle x', \log |z'| \rangle + \phi_F(z),$
 where $\log |z'|$ are orbit coordinates along $F$,  $\phi_F$ is
 the canonical
 \kahler potential for the subtoric variety defined by $F$ and $u_F$ is its Legendre transform.
  On the
 complement of $\bar{F}$ it is defined to be $+\infty$.

 \item When $z $ is a fixed point then $I^z(v) = 0$ and elsewhere
 $I^z(x) = \infty$.

\end{itemize}
 \end{theo}

The   local asymptotics of Lemma \ref{pcalLem} (due to  \cite[Lemma 5.1]{ZZ18}) are derived from this large deviations principle.

\subsection{\label{BERNSTEINSECT} Bernstein polynomials and associated measures}

One approach to entropy of the measures $\mu_k^z$ is to recognize their relation to Bernstein polynomials \cite{Z09}.
The 
Bernstein polynomials of a continuous function $f \in C(\overline{P})$ of a general toric \kahler manifold are quotients \begin{equation}
\label{QUOTIENT} B_{h^k}(f)(x) = \frac{\ncal_{h^k}
f(x)}{\Pi_{h^k}(\mu_h^{-1}(x), \mu_h^{-1}(x))}
\end{equation} of a {\it numerator polynomial} $\ncal_{h^k} f(x)$
by  the denominator $\Pi_{h^k}(z,z)$ with $\mu_h(z) = x$. Here,
$\mu_h$ is the moment map associated to the \kahler form
$\omega_h$ associated to $h$, and 
$$
\begin{array}{lll}  
   \ncal_{h^k} f(x)  &&=  \sum_{\alpha \in kP \cap \Z^m}
 f(\frac{\alpha}{k}) \frac{ e^{k
  \left(u_{\phi} (x) + \langle
\frac{\alpha}{k} - x, \nabla u_{\phi}(x) \rangle \right)}
}{Q_{h^k}(\alpha)}
 . \end{array}
 .$$
 The numerator polynomials as well as the denominator   admit
complete asymptotic expansions: 
  there
exist differential operators $\ncal_j$,   such that
$$ \ncal_{h^k}(f)(x)
\sim  \frac{k^m}{\pi^m}\left(   f(x)  + k^{-1} \ncal_1 f(x)  + \cdots \right),$$ where the
operators $\ncal_j$ are computable from the Bergman kernel
expansion for $\Pi_{h^k}(z,z)$. In particular,
$$\ncal_1 f(x) = \frac{1}{2} \left( f(x) S(\mu_h^{-1}(x) ) + \nabla
\mu_h (\mu_h^{-1}(x)) \cdot \nabla^2 f(x) \right), $$
where $S(z)$ is the scalar curvature of the \kahler metric $\omega_h$.
Combining the asymptotics of the numerator and denominator produces the asymptotics for $f \in C^{\infty}(\overline{P})$.
\begin{equation} \label{BERN}  B_{h^k}(f)(x) = f(x) + \lcal_1 f(x) k^{-1} + \lcal_2 f(x) k^{-2} + \cdots + \lcal_m f(x) k^{-m}
+ O(k^{- m - 1}), \end{equation} in $C^{\infty}(\bar{P})$, where $\lcal_j$ is
a differential operator of order $2j$ depending only on curvature
invariants of the metric $h$; the expansion may be differentiated
any number of times.

The relevance of Bernstein polynomials to the measures $\mu_k^z$ is the following easily verifiable formula \cite{Z09}: if 
$x = \mu_{\phi}(z)$ and
let $h = e^{- \phi}$, then 
 $$ \begin{array}{lll} B_{h^k}
  f(x) &= & \int_P f(y) d\mu^z_k(y).\end{array}.$$
   It follows that, for any $f \in C(\overline{P})$, 
\begin{equation}\label{WEAK}  \lim_{k \to \infty} \int_P f(y) d\mu_k^z(y) = f(\mu(z)).  \end{equation}

\section{Proof of Theorem \ref{H} on  entropy asymptotics }
In this section, we prove Theorem \ref{H}. Since it is a rather technical calculation, we first give a detailed outline using prior results on the large deviations principle reviewed in Section \ref{LDSECT}, and 
on Bernstein polynomials (reviewed in Section \ref{BERNSTEINSECT}). We then give a self-contained proof in Section \ref{SECONDSECT}. The outline is quite
detailed and helps as a guide to the self-contained proof.

\subsection{Sketch of proof}  The entropy of $\mu_k^z$ is given explicitly in \eqref{ENTFORM}.
The weights $ \pcal_{h^k}(\alpha, z)$ \eqref{PHK} of the dilate
$\mu_k^{z,1}$ admit pointwise asymptotic expansions in Lemma \ref{pcalLem}. We assume that $\mu(z) \in P^o$, so that
$ I^z(x) = u_0(x) -
\langle x, \log |z| \rangle + \phip (z).$  Unravelling the logarithm in \eqref{ENTFORM}  gives,  
\begin{equation}\label{H2}  \begin{array}{lll} H(\mu_k^z) & = &- \frac{1}{\Pi_{h^{k }}(z,z)} \sum_{\alpha \in k P \cap \Z^m} 
 \pcal_{h^k}(\alpha, z) \left( \log \pcal_{h^k}(\alpha, z) - \log \Pi_{h^{k }}(z,z) \right).
 \end{array} \end{equation}
 By \eqref{TYZ}, 
 \begin{equation} \label{1} (I) \;\;\;
   \frac{1}{\Pi_{h^{k }}(z,z)} \sum_{\alpha \in k P  \cap \Z^m} 
 \pcal_{h^k}(\alpha, z) \log \Pi_{h^{k }}(z,z) \sim   \log (a_0 k^m) + O(\frac{1}{k}) = m \log k + O(\frac{1}{k}). \end{equation}
 
 Therefore it suffices to determine the asymptotics of the first term of \eqref{H2}, 
 $$- \frac{1}{\Pi_{h^{k }}(z,z)} \sum_{\alpha \in k P  \cap \Z^m} 
 \pcal_{h^k}(\alpha, z) \log \pcal_{h^k}(\alpha, z) . $$
 By Lemma \ref{pcalLem}, 
 \begin{equation}\label{LOG} \log \pcal_{h^k}(\alpha, z) = \log (k^{m/2} (2\pi)^{-m/2}) + \log |\det  {\Hess}(u_{\varphi}(\mu_h(z))|^{\half}  - k I^z(\frac{\alpha}{k}) + O(1/k)). \end{equation}
 Since the first two terms are independent of $\alpha$, we obtain a second term,
 \begin{equation} \label{2} \begin{array}{l} (II) \;\;  - \frac{1}{\Pi_{h^{k }}(z,z)} \sum_{\alpha \in k P  \cap \Z^m} 
 \pcal_{h^k}(\alpha, z) \left(\log (k^{m/2} (2\pi)^{-m/2}) + \log |\det  {\Hess}(u_{\varphi}(\mu_h(z))|^{\half} \right) \\ \\= - \log (k^{m/2} (2\pi)^{-m/2}) - \log |\det  {\Hess}(u_{\varphi}(\mu_h(z)))|^{\half} + O(\frac{1}{k}). \end{array} \end{equation}

 Thus, we are left with the third term,
 \begin{equation} \label{3}  \begin{array}{l} (III) \;\;  \frac{1}{\Pi_{h^{k }}(z,z)} \sum_{\alpha \in k P  \cap \Z^m} 
 \pcal_{h^k}(\alpha, z)  (k I^z(\frac{\alpha}{k})).  \end{array} \end{equation} 
We obtain  asymptotics for this term using the asymptotics of   Bernstein polynomials as reviewed above . To make this connection, 
we  define a function $f_z$
so that $$f_z(\frac{\alpha}{k}) =  I^z(\frac{\alpha}{k}). $$
Both sides extend with no complication from the lattice points $\frac{\alpha}{k}$ to all $x \in P^o$. 
By Theorem \ref{LDINTRO}, it  follows that term (III) is, up to errors of order $O(\frac{1}{k})$,  the Bernstein polynomial
for  $$f_z(x) =   I^z(x) = u_0(x) -
\langle x, \log |z| \rangle + \phip (z). $$
Note that since $u$ and $\phi$ are Legendre transforms, one has
$$u(x) + \phi(\rho) = \langle x, \rho \rangle, \;\; x = \mu(e^{\rho}). $$ By \eqref{BERN},
 the leading  term in the asymptotic expansion of $(III)$  is $0$. Since this term is multiplied by $k$, this signals
 that $(III) = O(1)$. Since $f(\mu(z)) = 0$, the  leading order asymptotics is given by the second term, 
 \begin{equation} \label{3.0} (III.0) \;\;\ncal_1 f(x) = \frac{1}{2} \left( f(x) S(\mu_h^{-1}(x) ) + \nabla
\mu_h (\mu_h^{-1}(x)) \cdot \nabla^2 f(x) \right) =  \nabla
\mu_h (\mu_h^{-1}(x)) \cdot \nabla^2 I^z(\mu(z)) . \end{equation}
However, $\nabla^2 I^z = \nabla^2 u_0$, so the last term is $\rm{Tr}(I_m) = m$. 
 

Adding the contributions of \eqref{1}-\eqref{2}-\eqref{3.0} gives \begin{equation} \label{FINALb} \begin{array}{lll} H(\mu_k^z) & = &  (I) + (II) + (III.0) \\ && \\& = & 
  \log ( k^m)   - \log (k^{m/2} (2\pi)^{-m/2}) - \log |\det  {\Hess}(u_{\varphi}(\mu_h(z))|^{\half}  +  \frac{m}{2}  +  O(\frac{1}{k}) \\&&\\& = &
\frac{m}{2}(1 + \log (2 \pi k)) - \log |\det  {\Hess}(u_{\varphi}(\mu_h(z))|^{\half}  +  O(\frac{1}{k}) \end{array}, \end{equation}
agreeing with the formula of Theorem \ref{H}.

\begin{rem} Above, we used that $a_0 = 1$ to simplify the first term. \end{rem}

\subsection{\label{SECONDSECT} A more detailed proof} We now give a more detailed proof without using prior results on Bernstein polynomials.

Let $ \tilde Q_k(y):= e^{k u_0(y)} \int_P e^{k( u_0(x) - u_0(y) + < \nabla u_0 (x), y - x >)} dx$

\begin{lem}
For all neighborhoods of the boundary of the polytope U we have a uniform equivalent outside of $U$,
$ \tilde Q_k(y) = (2\pi)^{m/2} \lvert \det \nabla^2 u_0 (y) \lvert^{-1/2}  k^{m/2} e^{k u_0(y)} ( 1 + \frac{ m_k(y)}{k^{1/2-\epsilon}} ) $ 
where $ \sup_{k \in \mathbb{N}, y \in P-U} \lvert m_k(y) \lvert < \infty $
\end{lem}

\begin{proof}
let $f_y(x) =  u_0(y) - u_0(x) + < \nabla u_0 (x), x -y  >$. It is a positive function that attain 0 only once in $y$ and whose Hessian is $\nabla^2 u_0\lvert_y$ at $y$.

\begin{align} \tilde Q_k(y) &=  e^{k u_0(y)} \int_P e^{-k f_y(x)} dx \\ &= e^{k u_0(y)}  ( \int_{B(y,\delta_k)} e^{-k (\nabla^2 u_0\rvert_y(x-y , x-y) + \delta f(x))}  dx + Ke^{-k \inf_{P - B(x,\delta_k)}f_y(x)}). 
\end{align}

Let $m_y^\delta := \frac{\inf_{P - B(x,\delta)}f_y(x)}{2 \delta^2}$ and $M_y^\delta := \frac{f_y(x) - \nabla^2 u_0(x-y , x-y)}{6 \delta^3}$

We have two bounds on $\tilde Q_k(y)$
$$\begin{array}{lll} 
\int_{B(y,\delta_k)} e^{-\frac{k}{2} \nabla^2 u_0(x-y , x-y) - k \delta_k^3 M_y^{\delta_k} } dx 
\leq \tilde Q_k(y)e^{-k u_0(y)} 
& \leq & \int_{B(y,\delta_k)} e^{-\frac{k}{2} \nabla^2 u_0(x-y , x-y) + k \delta_k^3 M_y^{\delta_k} }  dx \\&& \\
 &+&  \Vol(P)e^{-k \delta_k^2 m_y^{\delta_k}}
\end{array}$$

Now changing the variables in the integral leads to

$$\begin{array}{lll} 
\int_{B(0,\sqrt{k}\delta_k)} e^{-\frac{1}{2} \nabla^2 u_0(z,z) - k \delta_k^3 M_y^{\delta_k} } k^{m/2}dz 
& \leq & \tilde Q_k(y)e^{-k u_0(y)}
\leq \int_{B(0,\sqrt{k} \delta_k)} e^{-\frac{1}{2} \nabla^2 u_0(z,z) + k \delta_k^3 M_y^{\delta_k} }  k^{m/2}dz \\ &&\\ && \ + \ \Vol(P)e^{-k \delta_k^2 m_y^{\delta_k}}
\end{array}$$

In order for the whole term to converge we need to choose $\delta_k$ to carefully. If we choose $\delta_k = \epsilon k^{-\alpha}$ with $\alpha \in ( \frac{1}{3} , \frac{1}{2} ) $ we'll obtain a exponential rate of convergence. 

More over as $k$ goes to infinity, $ \int_{B(0,\sqrt{k} \delta_k)} e^{-\frac{1}{2} \nabla^2 u_0(z,z)} = \sqrt{\det 2\pi (\nabla^2 u_0)^{-1}} ( 1 + N_k (\sqrt{k} \delta_k)^{m-2} e^{-\frac{(\sqrt{k} \delta_k)^2}{2}}) $ with $(N_k)_{k\in \mathbb{N}}$ a bounded sequence.

Now we have the following sandwich :

\[
\begin{split}
e^{-k \delta_k^3 M_y^{\delta_k}}( 1 + & N_k (\sqrt{k} \delta_k)^{m-2} e^{-\frac{(\sqrt{k} \delta_k)^2}{2}}) \leq \\
& \frac{Q_k(y) \sqrt{\det \nabla^2 u_0 (y)}}{(2\pi)^{\frac{m}{2}}  k^{m/2} e^{k u_0(y)}}
\leq e^{k \delta_k^3 M_y^{\delta_k}}( 1 + N_k (\sqrt{k} \delta_k)^{m-2} e^{-\frac{(\sqrt{k} \delta_k)^2}{2}} + N_k k^{-\frac{m}{2}} e^{-k \delta_k^2 m_y^{\delta_k}}) \\
\end{split}
\]

The first vanishing term comes from the term in $ e^{k \delta_k^3 M_y^{\delta_k}} $
The final equivalent is of the form : 
\[
\frac{Q_k(y)\sqrt{\det \nabla^2 u_0 (y)}}{(2\pi)^{\frac{m}{2}}  k^{m/2} e^{k u_0(y)}} = 1 + c_k^y k \delta_k^3 = 1 + c_k^y k^{1 -3\alpha} = 1 + c_k^y k^{-\frac{1}{2}+\epsilon} 
\]

With $(c_k^y)_{k\in\mathbb{N}}$ a bounded sequence. In order to prove that $(c_k^y)_{k\in\mathbb{N}}$ is uniformly bounded over $P-U$, we just need to show that $m_y^\delta$ and  $M_y^\delta$ are uniformly bounded. For any neighborhood $U$ of the boundary of $P$, $P-U$ is a compact set where $u_0$ is $C^\infty$ and so where  $m_y^\delta$ and  $M_y^\delta$ are uniformly bounded.
\end{proof}
 
\subsection{Computation of the entropy}

Now that we proved this technical lemma, we'll use it and the asymptotics of the \Szego kernel $ \Pi(z,z) = k^m + O(k^{m-1})$ to obtain the following uniform asymptotic for the individual probabilities of the sequence of measures :
Let's take any neighborhood of the boundary $U$ such that $\mu(z)$ is in the interior of $P-U$. 
We have that $\forall \alpha \in P-U $ such that $k\alpha \in kP \cap \mathbb{Z}^m$ for a certain $k$. 
\[
\mu_k^z(\alpha) = \frac{ \lvert z^{\alpha} \rvert^2_{h^k}}{Q_k(\alpha) \Pi_k(z,z)} = \frac{\sqrt{\det \nabla^2 u_0 (y)}}{{(2\pi k)^{-m/2}}} e^{-k I^z(\alpha)} \bigg(1+\frac{c_\alpha} {k^{1/2 - \epsilon}} \bigg)
 \]
 
 With $ \lvert c_{\alpha} \lvert \leq M $ and $M$ only depending on $U$.
 
 Let's split the calculation in two :
\[
H(\mu_k^z) = - \sum_{p \in U} \mu_k^z(p) \log (\mu_k^z(p)) - \sum_{p \in P-U} \mu_k^z(p) \log (\mu_k^z(p))
\]

\begin{lem}
The first term goes to zero
\end{lem}

\begin{proof}
 Let $\nu_k := \frac{\mu_k^z(\cdot)\textbf{1}_{\cdot \in U}}{\mu_k^z(U)}$. Then 
\[
- \sum_{p \in U} \mu_k^z(p) \log (\mu_k^z(p)) =  - \mu_k^z(U) \log ( \mu_k^z(U) ) +  \mu_k^z(U) H(\nu_k)
\]
Except that $\nu$ is concentrated on a finite number of points that is equal to $ \lvert U \lvert k^{m} + o(k^{d}) $, meaning that $H(\nu_k) \leq d \log(k) + constant$ and $\mu_k^z(U)$ decrease exponentially due to the LDP. This implies that the result.
$$
- \sum_{p \in U} \mu_k^z(p) \log (\mu_k^z(p)) \rightarrow_{k\rightarrow \infty} 0 
$$
\end{proof}

 We need to compute the second term $H'(\mu_k^z)$. Let's split the sum again in four parts. We'll use the notation $ P_k^U := \Big( kP \cap \mathbb{Z}^m \Big) / k - U $.
  
\begin{align}
H'(\mu^z_k) :&= \underset{\alpha \in P_k^U}\sum \mu_k^z(\alpha) \log((2\pi k)^{m/2})  \\
&- \underset{\alpha \in P_k^U}\sum \mu_k^z(\alpha) \log(\sqrt{\det \nabla^2 u_0(\alpha)})  \\
&+ \underset{\alpha \in P_k^U}\sum \mu_k^z(\alpha) k I^z(\alpha) \\
&- \underset{\alpha \in P_k^U}\sum \mu_k^z(\alpha) \log\bigg(1+\frac{c_\alpha} {k^{1/2 - \epsilon}} \bigg)
\end{align}

Then trivially we have that $ \lvert(4)\lvert < \log \bigg( 1 + \frac{M} {k^{1/2 - \epsilon}} \bigg) \underset{k\rightarrow \infty}\rightarrow 0$ and that $ (1) =  \frac{m}{2}\log(2\pi k) + \text{o}(1) $.

For $(2)$ we need to notice that $u_0$ is strictly convex on the interior of $P$, so it's stricly concave on $P-U$. $\log(\lvert \det \nabla^2 u_0(\alpha) \rvert)^{1/2}$ is a consequently a continuous bounded function on $P-U$ and the LDP implies that the term converges to $\log(\lvert \det \nabla^2 u_0(\mu(z)) \rvert^{1/2})$.

The only difficult term to compute is the third, which we'll denote $A_3$.
 
 \begin{lem}
 $\lim_{k \to \infty} A_3 = \frac{m}{2}$
 \end{lem}
 
 \begin{proof}
Let $K:= -\nabla^2u_0(\mu(z))$. For a $ \delta $ arbitrary small, we have : 
\begin{itemize}
\item Outside of $ B_\delta(x_0)$ we have $ I^z > \epsilon_1(\delta) \delta^2 $ with $\epsilon_1(\cdot)$ a strictly positive function with a strictly positive lower bound
\item Inside of $ B_\delta(x_0)$ we have $ f_k(\alpha) = \lvert \det \nabla^2 u_0(\alpha) \rvert^{1/2} \Big(1+\frac{c_\alpha^k} {k^{1/2 - \epsilon}} \Big)  = \lvert \det K \rvert^{1/2} (1 + \epsilon_2^k(\alpha-x_0))$ such that $ \lVert \epsilon_2^k (\cdot) \rVert_\infty < \epsilon_{2,\delta}^k $, with $\epsilon_{2,\delta}^k$ increasing with $\delta$ and decreasing with $k$ such that it vanishes as $k\rightarrow \infty$ and $\delta \rightarrow 0$.
\item Inside of $ B_\delta(x_0)$ we also have that $ I^z(x_0+\delta x) = \frac{1}{2}K(\delta x, \delta x) + \epsilon_3(\delta x)\lVert \delta x \rVert^3$ with $ \lVert \epsilon_3 (\cdot) \rVert_\infty < \infty $
\end{itemize}

We treat the $\epsilon_1$ as an increasing positives functions of $\delta$ that vanish in 0 and $\epsilon_2^k$ as an increasing with $\delta$ and decreasing with $k$ positive function that vanish in 0 as $k\rightarrow \infty$.
We then have :

\begin{align}
\frac{A_3 }{\sqrt{\det K}(2\pi)^{-m/2}}   &=  \frac{(2\pi)^{-m/2}}{\sqrt{\det K}}\underset{\alpha \in P_k^U}\sum \mu_k^z(\alpha) k I^z(\alpha) \\
&= \underset{\alpha \in P_k^U \cap B_\delta(x_0)}\sum k^{-m/2} \sqrt{\frac{\lvert \det \nabla^2 u_0(\alpha) \rvert}{\det K}} \bigg(1+\frac{c_\alpha} {k^{1/2 - \epsilon}} \bigg) e^{-k I^z(\alpha)} k I^z(\alpha) + O(k^{1+m/2} e^{-k \epsilon_1(\delta) \delta^2} ) \\
&= k^{-m/2} \underset{\delta x \in (P_k^U-x_0) \cap B_\delta(0)}\sum e^{-\frac{k}{2} K(\delta x,\delta x)} \frac{k}{2} K(\delta x,\delta x) ( 1 + O(\epsilon_2^k) + O(k \delta^3))  +  O(k^{1+m/2} e^{-k \epsilon_1(\delta) \delta^2} ) \\
&\text{ with the two $O$ being uniformly bounded over $(P_k^U-x_0) \cap B_{\delta_0}(0)$. Now let's scale up the sum.} \\
&= k^{-m/2} \underset{\delta x \in \sqrt{k}(P_k^U-x_0) \cap B_{\sqrt{k}\delta}(0)}\sum e^{-\frac{1}{2} K(\delta x,\delta x)} \frac{1}{2} K(\delta x,\delta x) ( 1 + O(\epsilon_2^k) + O(k  \delta^3))  +  O(k^{1+m/2} e^{-k \epsilon_1(\delta) \delta^2} )
\end{align}

Now the set $ P_\delta^k := \sqrt{k}(P_k^U-x_0) \cap B_{\sqrt{k}\delta}(0)$ is for small enough $\delta$ and interior $x_0$ the set $\frac{1}{\sqrt{k}}(\mathbb{Z}^m- x_0) \cap B_{\sqrt{k}\delta}$. If we choose a specific sequence of $\delta_k = \frac{\epsilon}{k^\gamma}$ with $\gamma \in (\frac{1}{3},\frac{1}{2})$, the series converges as $k$ goes to infinity and all the $O$s vanish from the limit.

The series is a truncated step function approximation of the following integral
\[
\int_{\mathbb{R}^m} \frac{K(x,x)}{2}e^{-\frac{K(x,x)}{2}} dx = \sqrt{\det{\frac{2\pi}{K}}} \frac{m}{2}
\]
Which is Riemann integrable so we don't need further arguments.
Finally,
\[ \lim_{k \to \infty} A_3 = \frac{m}{2 a_0 (2\pi)^m}\]

\end{proof}

By binding all the pieces together, we obtain the following asymptotics for the entropy of the measures $\mu_k^z$ :

\[
H(\mu_k^z) \underset{k\rightarrow \infty}{=} \frac{m}{2}\log(2 \pi k) - \frac{1}{2}\log(\lvert \det \nabla^2 u_0(x_0) \rvert) + \frac{m}{2} + o(1)
\]

This concludes the proof of Theorem \ref{H}.

\subsection{Point $z$ for which  $\mu_k^z$ has asymptotically maximal entropy: Proof of Theorem \ref{MAX}}
We now consider the point $x  = \mu(z)$ for which the measure $\mu_k^z$ has asymptotically  maximal entropy within the family $\{\mu_k^z\}.$ 
For Fano \kahler manifolds, we prove that there exists a unique point $x = \mu(z)$ at which $\mu_k^z$ has asymptotically maximal entropy. For
Fano \kahler manifolds with \kahler-Einstein metric, we prove that the unique point $x \in P$ is the point $x =0$.

\begin{proof}
As mentioned in the introduction, it follows from 
Theorem \ref{H} that the points $z$ such that $\mu_k^z$ has asymptotically maximal entropy are points where the Ricci potential is maximal. 
Due to the inverse relation of $ \det i \ddbar \phi$ and $L(x) = \log(\lvert \det \nabla^2 u_0(x_0) \rvert)$ 
\eqref{LDEF}, points where the Ricci potential is maximal  lie in the inverse image under the moment map $\mu_h$
of points $x$  for which $L(x) $  is minimal (see Theorem \ref{H2TH}). Since $u$ is convex,  we may remove the absolute value.

It is proved in \cite[Theorem 2.8]{Ab98} that  $L(x)$  is a smooth function on the interior of $P$ and $L(x) \uparrow \infty$ as $
x \to \partial P$. Consequently, $u$ has a global minimum  which lies in the interior of $P$.  This proves existence for all toric \kahler manifolds.

If we assume that  Ricci curvature $\rm{Ric}(\omega) $ is positive definite, then the  minimum
is unique. Indeed, by  \eqref{RICCI}, positive Ricci curvature   is equivalent to $L(x)$
 \eqref{LDEF} being a strictly convex function.  Since $L(x) \uparrow \infty$ as $
x \to \partial P$, $L(x)$ is proper and strict convexity implies that the minimum of $u$ is unique. We state the result as the following, 
\begin{lem} \label{RICPOTLEM} If $(M, \omega)$ is a toric Fano manifold of positive Ricci curvature, then there is a unique point $\rho_0$ in the open orbit of maximal
entropy, corresponding to a unique point $x_0 = \mu_h(e^{\rho_0/2}) \in P. $ \end{lem}

We now assume further that $\omega$ is a toric \kahler-Einstein metric. In that case,  the
equation $\rm{Ric}(\omega) = a  \omega$ (for some $C > 0$) implies that  there exists a constant vector $\vec b$ and $c \in \R$ so that  \begin{equation}
\label{KEP} -\log \det (i \ddbar \varphi)(\rho) = a \varphi(\rho) + \vec b \cdot \rho + c. \end{equation} Indeed,
$\ddbar (-\log \det (i \ddbar \varphi)(\rho) - a \varphi(\rho)) =0$ and therefore the difference potential  is a toric pluriharmonic function, hence a linear function.
By Lemma \ref{RICPOTLEM}, we get
\begin{lem} \label{KERICPOTLEM} If  If $(M, \omega)$ is a toric \kahler-Einstein  Fano manifold of positive Ricci curvature, then in the gauge \eqref{KEP},
there is a unique critical point in the open orbit where \begin{equation} \label{CRIT} \nabla \varphi(\rho) = \mu_h(e^{\rho}) =  - \frac{1}{a} \vec b. \end{equation}
\end{lem}
Alternatively, we may write  the critical point equation in terms of   the symplectic potential $u$ and the function $L(x)$ \eqref{LDEF}. 
The \kahler-Einstein equation \eqref{KEP} for the potentials then changes to,
\begin{equation} \label{KESP} L(x) = a \phi(e^{\rho_x}) + \vec b \cdot \rho_x + c, \;\; \mu_h(\rho_x) = x. \end{equation} Since $u = \lcal \phi$ (Legendre transform),  and $\nabla u(x) = \mu_h^{-1} (x)$ one has
 $u(x) = x \cdot \rho_x - \phi(e^{\rho_x})$ and so \eqref{KESP} simplifies to,
 \begin{equation} \label{KESP2} L(x) = a (x \cdot \nabla u(x)  - u(x)) + \vec b \cdot \nabla u(x)  + c = (a x + \vec b) \cdot \nabla u - a u(x)  + c.\end{equation} 
 
 As mentioned in Section \ref{BACKGROUNDOO}, toric \kahler potentials are not unique because of the gauge symmetries. For instance, one may add a linear function of $\rho$ to obtain an equivalent potential. 
The shift of gauge by an affine function results in a translation of the Delzant polytope. 
According to \cite[Definition 3.6]{M87}, there exists an  open orbit toric \kahler potential $\phi$ so that 
 so that \begin{equation} \label{MAB} \det D^2_{\rho} \varphi = e^{- \varphi}, \;\; e^{-\varphi} \prod_{j=1}^m \frac{i dz_j \wedge d \bar{z}_j }{|z_j|^2}\in C^{\infty}(M, \Omega), \end{equation} where $z_j$ are open orbit coordinates
  (in  $(\C^*)^m$ and  $C^{\infty}(M, \Omega) $ is the space of smooth volume forms. We refer to \eqref{MAB} as the Mabuchi or \kahler-Einstein gauge. In the gauge \eqref{MAB}, the the potential satisfies,
   \begin{equation} \label{GOODGAUGE} \varphi = - \log \det (i \ddbar \varphi)(\rho) \; (\vec b = 0). \end{equation}

Combining with Lemma \ref{KERICPOTLEM} gives, \begin{lem} \label{GAUGE}  In the  gauge \eqref{GOODGAUGE}   for the  \kahler potential,   the unique point of maximal entropy solves the critical point equation,
  $$\nabla_{\rho}  \varphi(\rho) = 0, \;\; (\iff \mu(e^{\rho/2}) = 0). $$\end{lem}
  
  It follows that in this gauge, the point $x_0 \in P$ for which $\mu_k^{z_0}$ has asymptotically maximal entropy is the origin $0 \in P$. 
 
 This completes the proof of Theorem \ref{MAX}.
\end{proof}

\begin{rem} According to \cite{M87}, 
  the origin  in the gauge \eqref{MAB}  is the center of mass of the polytope $P$. This follows from the facts that, 
  by  \cite[Corollary 5.5]{M87},  the so-called Futaki invariant ${\bf a}_P = 0$, and by \cite[Theorem 9.2.3]{M87}, that the center
 of mass of $P$ is $0$. 

   One may directly prove that  the
center of mass in the \kahler-Einstein gauge equals $0$,  i.e. the gauge $\varphi$
such that $- \log \det i \ddbar \varphi = a \varphi$, using the moment
map change of variables $\mu_h(e^\rho) = \nabla_\rho \varphi (\rho) =
x(\rho)$, which gives $$
x_{\text{mass}} = \frac{1}{\Vol(P)} \int_P x dx = \frac{1}{\Vol(P)}
\int_{\R^n} \nabla_\rho \varphi (\rho) \det { \nabla_\rho^2 \varphi } \
d\rho. $$
Using the fact that $\det i\ddbar \varphi (e^\rho) = \det \nabla_\rho^2
\varphi d\rho = e^{- \varphi(\rho)} d\rho $ we get that \begin{equation} \label{CMASS} 
x_{\text{mass}} = -\frac{1}{\Vol(P)} \int_{\R^n} \nabla_\rho (e^{-
\varphi }) d\rho \end{equation}
By \eqref{OOKP},
 $ \varphi (\rho) = \max_{\alpha \in P \cap \Z^m}\;  \rho \cdot \alpha + O(1)
   $ as $\rho \to \infty$. As long as  $ \max_{\alpha \in P \cap \Z^m}\;  (\rho \cdot \alpha)
 \geq \epsilon |\rho| $ for some $\epsilon >0$ and large $\rho$, $e^{- \varphi}$ is
rapidly decreasing and one can integrate by parts  in \eqref{CMASS} to prove
that $x_{\text{mass}}=0$. The lower bound is true if and only if $0 \in P^o$. 
It must be the case that $\max_{\alpha \in P \cap \Z^m}\;  \rho \cdot \alpha \geq \epsilon |\rho|$ for large $\rho$; otherwise,
 there exists
direction in which $\varphi$ does not go to $+\infty$ implying that $$
\Vol(P) = \int_P dx = \int_{\R^n} e^{- \varphi} d \rho = \infty $$
which is a contradiction thus proving the result.
  \end{rem}

  Some further references for the existence of the potential satisfying \eqref{MAB} are \cite{WZ04, W15}. 
 The formula \eqref{KESP2}  agrees with  \cite[(2.18)]{WZ04}.  They define the parameters
 $c_{\ell}$  by,
 $$\int_P y_{\ell} \exp \{\sum_{\ell =1}^n c_{\ell} y_{\ell} \} dy = 0.$$ 
if $M$ admits a \kahler-Einstein metric then $\vec c = 0$; in terms of our notation, $\vec c$ is the center of mass of $P$. See also \cite[Page 3615]{W15}
and  \cite[Section 3.3]{D08}. 

Let us check the identities in the simplest case of $\CP^1$ with Fubini-Study metric. If we  choose the gauge $u_{FS}(x) =x \log x + (1-x) \log (1-x)$  in which $P = [0,1]$,  
 then  $\log (u_{FS}''(x))^{-1} =  \log x (1-x)$, 
$b=1, a = -2.$
Similar equations hold for $\CP^m$. If we choose a gauge for which $u_{r}(x) = \half \left((r^2 + x) \log (r^2+ x) + (r^2 - x) \log (r^2 - x)\right)$ 
and  $P = [-r^2,r^2]$, then $b = 0$.

     \section{Convolution powers and toric measure: Proof of Theorem \ref{INVCONV}}
     
     \subsection{Proof of Proposition \ref{EQUIV}} 
     
     First, we prove Proposition \ref{EQUIV}.  For the reader's convenience we  recall that it says that,
     for any \kahler manifold $(M, \omega, J)$, the following are equivalent: \begin{enumerate}
\item $\rm{Hilb}_k(h)$ is {\it balanced} for all $k$, i.e. there exist constants $C_k$ so that the density of states  $\Pi_{h^k}(z) = C_k$ for all $k$.  \bigskip

\item  $\Pi_{h^k}(z,w)= A_k (\Pi_{h^1}(z,w))^k$, where 
$$ \frac{\dim H^0(M, L^k)}{\rm{vol} (M, \omega)}  = A_k  \left(\frac{\dim H^0(M, L)}{\rm{vol} (M, \omega)} \right)^k. $$

\end{enumerate}




\begin{proof}$(1) \implies (2)$.    If $\Pi_{h^k}(z,z) = C_k$ for some constant $C_k$, then necessarily $C_k =  \frac{\dim H^0(M, L^k)}{\rm{vol} (M, \omega)} $. If (1) holds, then the constant $C_k$ in (1)  is given by this formula for all $k$. 
We then define $A_k$ by   $ A_k = \frac{C_k}{ C_1^k}$. Thus, $C_k$ and $A_k$ are uniquely determined by the assumption (1) and by definition of $A_k$
we have  $\Pi_{h^k}(z,z)= A_k (\Pi_{h^1}(z,z))^k$. This equation
holds along the  totally real submanifold $\{(z, \bar{z}): z \in M\} \subset M \times \bar{M}$, where we identify the complexification of
$M$ with $M\times \bar{M}$. Since $\Pi_{h^k}(z,w)$ is holomorphic, the equation implies $\Pi_{h^k}(z,w)= A_k (\Pi_{h^1}(z,w))^k$ for
all $(z, w) \in M\times \bar{M}$, i.e. $(1) \implies (2). $

$(2) \implies (1)$. Conversely assume (2).  To prove (1) it suffices to prove that $\Pi_{h^1}(z,z) = C_1$. If this is false, we introduce constants $\alpha < 1, \beta > 1$ and consider the sets
$M_- = \{z: \Pi_{h^1}(z,z) < \alpha  C_1\}$ and $M_+ =  \{z: \Pi_{h^1}(z,z) >  \beta C_1\}.$ If $\Pi_{h^1}(z,z) = C_1$ is false,   $M_{\pm}$  must be non-empty open sets for
some $\alpha < 1, \beta > 1$.  Assuming (2),
we have $\Pi_{h^k}(z,z) < A_k C_1^k \alpha^k $ in $M_-$ and  $\Pi_{h^k}(z,z) > A_k C_1^k  \beta^k$ in $M_+$. But $A_k C_1^k = C_k=  \frac{1}{\rm{Vol}(M)} \dim H^0(M, L^k),$ and  standard Bergman kernel asymptotics give  $\Pi_{h^k}(z, z)  \simeq \frac{1}{\rm{Vol}(M)} \dim H^0(M, L^k) + o ( \dim H^0(M, L^k) )$.
We then  get the contradiction that $\frac{1}{\rm{Vol}(M)} \dim H^0(M, L^k)  < \alpha^k \frac{1}{\rm{Vol}(M)} \dim H^0(M, L^k)$ in $M_-$ and
$ \frac{1}{\rm{Vol}(M)} \dim H^0(M, L^k)  > \beta^k \frac{1}{\rm{Vol}(M)} \dim H^0(M, L^k)$, concluding the proof.

\end{proof}

\begin{rem} 
 We note that $\Pi_{h^k}(z,w) $ is holomorphic in $z$ and anti-holomorphic in $w$. The density of states
$\rho_k= \Pi_{h^k}(z,z)$ is the metric contraction of the diagonal values of $\Pi_{h^k}(z,w) $ by the metric $e^{- k \phi(z)}$ and therefore
is not the restriction of a holomorphic function on the complexification $M \times \bar{M}$ to the anti-diagonal.  Hence, $\Pi_{h^k}(z,z) = C_k$
does not imply that $\Pi_{h^k}(z,w) = C_k$. But the equation $\Pi_{h^k}(z,z)= A_k (\Pi_{h^1}(z,z))^k$ is the restriction of a holomorphic equation
to the anti-diagonal and therefore extends to all of $M \times \bar{M}$.
\end{rem}

\subsection{Proof of  Theorem \ref{INVCONV}}

To prove Theorem \ref{INVCONV} we need to relate convolution powers of $\mu_k^z$ and the two conditions in Proposition \ref{EQUIV}.

Define the lattice path  `partition function': For $\alpha \in k P \cap \Z^m$,
\begin{equation} \label{PARTFUN} {\mathcal P}_k(\alpha)  := \sum_{(\beta_1,
\dots, \beta_k): \beta_j \in P, \beta_1 + \cdots + \beta_k =
\alpha }\;\; \prod_{j=1}^k \frac{1}{\qcal(\beta_j)}. \end{equation}

Then, we have

\begin{lem}\label{POWERS}  $$\Pi_{h^k}(z,z)
= A_k (\Pi_{h^1}(z,z))^k \iff \pcal_k   \qcal_k = A_k. $$ \end{lem}

\begin{proof}
Recall from \eqref{FNa} that 
\begin{equation}\label{FNb}F_{h^k}(z, w)=
\sum_{\alpha \in k P \cap \Z^m}  \frac{z^{\alpha}
\bar{w}^{\alpha}}{Q_{h^k}(\alpha)} \;,\end{equation}   where $Q_{h^k}(\alpha)$ is defined in \eqref{QFORM},
and that  $$\Pi_{h^k}(x, y ) =  F_{h^k}(z,w) e_L^k(z) \overline{e_L^k(w)}.$$
On the other hand, by definition of the partition function, we
also have
$$ F_{h^1}^k (x,y) = \sum_{\alpha \in k P  \cap \Z^m} {\mathcal
P}_k(\alpha)  z^{\alpha}
\bar{w}^{\alpha}
$$
If we contract the diagonal values of each kernel with the metric, the hypothesis of the Lemma gives,
$$F_{h^k}(z,z) = A_k (F_{h^1}(z,z))^k, $$
and comparing coefficients of the monomials completes the proof.

\end{proof}


Next, we evaluate the the Fourier transform $\fcal_{x \to \xi} (\mu_1)^{*k} $  of the  convolution powers of $\mu_1^z$. 

\begin{lem} We have,
$$ \fcal_{x \to \xi}   (\mu_1^z)^{*k}=   \frac{1}{(\Pi_{h^1}(z,z))^k}  \sum_{\alpha \in k P \cap  \Z^m} \pcal_k(\alpha) e^{i \langle \alpha, \xi \rangle}.
$$
\end{lem} 
\begin{proof} By definition of the partition function \eqref{PARTFUN}, 
$$\begin{array}{lll} \fcal_{x \to \xi}   (\mu_1^z)^{*k}   = (\fcal \mu_1^z)^k(\xi)  & = & 
 \left( \sum_{\alpha \in P} \frac{
\pcal_{h^1}(\alpha, z)}{\Pi_{h^1}(z,z)} e^{i \langle \alpha, \xi \rangle} \right)^k\\&&\\
& = &   \frac{1}{(\Pi_{h^1}(z,z))^k}  \sum_{\alpha \in k P \cap  \Z^m} \pcal_k(\alpha) e^{i \langle \alpha, \xi \rangle}
\end{array}$$

\end{proof}

The following Lemma is the main step in the proof of  Theorem \ref{INVCONV}.

\begin{lem} \label{CONVPOWERS} $\mu_k^z = (\mu_1^z)^{*k}$ for all $k$ and all $z\in M^o$  if and only if   \begin{equation} \label{ID} \frac{ \Pi_{h^k} (z,z)}{  (\Pi_{h^1}(z,z))^k}  =  A_k=   \pcal_k(\alpha) \qcal_k(\alpha), \;\; \forall k, \alpha \in kP \cap \Z^m, \forall z \in M^o, \end{equation}
where  $A_k$ is the constant  determined by Proposition \ref{EQUIV}. \end{lem}
\begin{proof} By the definition \eqref{PHK} of $\pcal_{h^k}(\alpha)$,

 $$\begin{array}{lll} \mu_k^z = (\mu_1^z)^{*k} & \iff & \sum_{\alpha \in kP  \cap \Z^m} \frac{
\pcal_{h^k}(\alpha, z)}{\Pi_{h^k}(z,z)} e^{i \langle \alpha, \xi \rangle} =  \left( \sum_{\alpha \in P} \frac{
\pcal_{h^1}(\alpha, z)}{\Pi_{h^1}(z,z)} e^{i \langle \alpha, \xi \rangle} \right)^k\\&&\\
& \iff & \frac{
\pcal_{h^k}(\alpha, z)}{\Pi_{h^k}(z,z)} = \frac{\pcal_{h^1}(\alpha, z)^k}{(\Pi_{h^1}(z,z))^k},  \;\;\;\;\forall \alpha \in k P \cap \Z^m\\&& \\ 
&\iff & 
 \frac{1}{\qcal_{h^k}(\alpha) \Pi_{h^k}(z,z)} = \frac{1}{(\Pi_{h^1}(z,z))^k} \pcal(\alpha),  \;\;\;\;\forall \alpha \in k P \cap \Z^m\\&&\\
 & \iff  & \frac{(\Pi_{h^1}(z,z))^k}{\Pi_{h^k}(z,z)} =  \pcal(\alpha) \qcal_{h^k}(\alpha),  \;\;\;\;\forall \alpha \in k P \cap \Z^m.

 \end{array}$$
 Equality is only possible if the left side is independent of $z$ and the right side is independent of $\alpha$. Therefore, both must
 be a constant, and the left side must be the constant $A_k$ determined by Proposition \ref{EQUIV}.
 
\end{proof}

We now  complete the proof of Theorem \ref{INVCONV}. First assume that $\mu_k^z = (\mu_1^z)^{*k}$ for all $k$ and all $z\in M^o$. 
Then by Lemma \ref{CONVPOWERS}, $ \frac{ \Pi_{h^k} (z,z)}{  (\Pi_{h^1}(z,z))^k}  =  A_k$ for all $k$ and $z \in M^o$) (hence for $z \in M$). 
Then, by Proposition \ref{EQUIV},   $\Pi_{h^k}(z,z)$ is a constant for all $k$, say $D_k$. The constant $D_k$ is determined by integrating and 
as in Lemma \ref{EQUIV},  $D_k =\frac{1}{Vol(M)} \dim H^0(M, L^k) $. By the 
 Bergman kernel expansion \eqref{TYZ}, $\omega$ is a CSC metric.
 
 Conversely suppose that $\rm{Hilb_k}(h)$ is balanced for each $k$, i.e. that  $\Pi_{h^k}(z) $ is constant. By Proposition \ref{EQUIV},   $ \frac{ \Pi_{h^k} (z,z)}{  (\Pi_{h^1}(z,z))^k}  =  A_k$ for all $k$ and $z \in M^o$ and by Lemma \ref{CONVPOWERS}, $\mu_k^z = (\mu_1^z)^{*k}$ for all $k$.

\subsection{Differential entropy of Gaussian measures on $H^0(M, L^k)$: Proof of Proposition \ref{HILBH} }

Proposition \ref{HILBH} is of a different nature from the preceding results, since it concerns Gaussian measures on $H^0(M, L^k)$ induced
by Hermitian metrics on $K$, rather than the toric measures $d\mu_k^z$. But it is related in that both concern entropies of probability measures
induced by \kahler metrics. The proof is rather simple, because we may reduce it to results of Donaldson on balanced metrics. 
    
\begin{proof} The entropy $H(\gamma_P | \gamma_I)$ of this
    Gaussian measure relative to that of the background is $- \log \det P.$

  In the case of a toric \kahler manifold, we may represent an inner product by the norming constants $Q_{h_k}(\alpha)$. 
 In fact the
   toric Gaussian measure is the product measure $$\prod_{\alpha \in k P  \cap \Z^m} \sqrt{Q_{h_k}(\alpha)} e^{- \langle Q_{h_k}(\alpha)^{-1} x, x \rangle} dx $$ Then
  $\det P = \prod_{\alpha \in k P  \cap \Z^m} Q_{h_k}(\alpha)$. It follows that the differential entropy of the associated Gaussian measure is
  \begin{equation} \label{HgammaQ} H(\gamma_{\vec Q_{h_k}}) = - \log \det \rm{Hilb}_k(h) =  -\sum_{\alpha \in k P  \cap \Z^m} \log Q_{h_k}(\alpha). \end{equation}

   Interestingly, \eqref{HgammaQ} is  the functional $\lcal$  introduced by Donaldson in \cite[(10)]{D05}. In \cite[Lemma 2]{D05} and \cite[Corollary 1]{D05} 
   it is proved that a metric is balanced if and only if it is a critical point of the functional $\wt{\lcal} = \lcal - \frac{d}{V} I $ on the space $\kcal$ of \kahler
   metrics in the fixed $(1,1)$ class. In fact, as explained there, $\delta \lcal$ vanishes for all $\delta \phi$ of integral zero if and only if the density
   of states $\Pi_{h^k}(z)$ is a constant. The second term $-\frac{d}{V} I$ is only to fix the undetermined constant in the \kahler potential and may be omitted
   if we work with global potentials on the open orbit.

     \end{proof}

\end{document}